\documentclass[11pt]{article}
\usepackage{fullpage}
\usepackage{authblk}
\usepackage[utf8]{inputenc} % allow utf-8 input
\usepackage[T1]{fontenc}    % use 8-bit T1 fonts
\usepackage{hyperref}       % hyperlinks
\usepackage{url}            % simple URL typesetting
\usepackage{booktabs}       % professional-quality tables
\usepackage{amsfonts}       % blackboard math symbols
\usepackage{nicefrac}       % compact symbols for 1/2, etc.
\usepackage{microtype}      % microtypography

\usepackage{natbib}
\usepackage{amsmath}
\usepackage{amsthm}
\usepackage{amssymb}
\usepackage{mathabx}
\usepackage{tikz}
\usepackage{xcolor}
\usetikzlibrary{arrows}

\usepackage[title]{appendix}

\usepackage{algorithm}
\usepackage{algorithmic}
\usepackage{hyperref}
\usepackage{bm}

\allowdisplaybreaks[4]

\usepackage{caption}
\usepackage{subcaption}

\newcommand{\median}{\mathrm{median}}

\newcommand{\mat}[1]{{#1}}
\newcommand{\vect}[1]{\mathbf{#1}}
\newcommand{\norm}[1]{\left\|#1\right\|}
\newcommand{\abs}[1]{\left|#1\right|}
\newcommand{\expect}[1]{\mathbf{E}\left[#1\right]}
\newcommand{\prob}{\mathbb{P}}

\newcommand{\holder}{H\"{o}lder }

\newcommand{\variance}{\textbf{Var}}

\newtheorem{thm}{Theorem}
\newtheorem{lem}{Lemma}
\newtheorem{cor}{Corollary}
\newtheorem{prop}{Proposition}

\newtheorem{defn}{Definition}

\newtheorem{rem}{Remark}

\renewcommand{\hat}{\widehat}
\renewcommand{\tilde}{\widetilde}
\renewcommand{\bar}{\overline}
\renewcommand{\check}{\widecheck}

\renewcommand{\triangle}{\Delta}

\title{Robust Nonparametric Regression under \\Huber's $\epsilon$-contamination Model}

\author[1]{Simon S. Du}
\author[1]{Yining Wang}
\author[2]{Sivaraman Balakrishnan}
\author[1]{Pradeep Ravikumar}
\author[1]{\qquad \qquad Aarti Singh}
\affil[1]{Machine Learning Department, Carnegie Mellon University}
\affil[2]{Department of Statistics, Carnegie Mellon University}

\begin{document}

\maketitle

\begin{abstract}
	\label{sec:abs}
	We consider the non-parametric regression problem under Huber's $\epsilon$-contamination model, 
in which an $\epsilon$ fraction of observations are subject to arbitrary adversarial noise.
We first show that a simple local binning median step can effectively remove the adversary noise and this median estimator  is minimax optimal up to absolute constants over the \holder function class with smoothness parameters smaller than or equal to 1.
Furthermore, when the underlying function has higher smoothness, we show that using local binning median as pre-preprocessing step to remove the adversarial noise, then we can apply any non-parametric estimator on top of the medians.
In particular we show local median binning followed by kernel smoothing and  local polynomial regression achieve minimaxity over \holder and Sobolev classes with arbitrary smoothness parameters.
Our main proof technique is a decoupled analysis of adversary noise and stochastic noise, which can be potentially applied to other robust estimation problems.
We also provide numerical results to verify the effectiveness of our proposed methods.
\end{abstract}

\section{Introduction}
\label{sec:intro}
Nonparametric regression has a wide range of applications in statistics and machine learning research~\citep{larry2006all,tsybakov2009introduction,friedman2001elements}.
In this paper we restrain ourselves to the fixed design setting, 
where design points $x_\vect{i} = \left[\frac{i_1}{p},\ldots,\frac{i_d}{p}\right]$, $i_1,i_2,\ldots,i_d=1,\cdots,p$ are evenly spaced on $[0,1]^d$ and in total we have $n = p^d$ design points.
In the standard setting, we observe
$
y_\vect{i} = f(x_\vect{i}) + \xi_\vect{i}
$
where $f:[0,1]^d\to\mathbb R$ is an underlying function to be estimated
and $\{\xi_\vect{i}\}$ are i.i.d.~noise variables.
The objective is to construct an estimate $\hat{f}$ that is close to $f$ under certain error metric like the mean square error:
$
\|\hat f-f\|_2^2 := \int_{[0,1]^d}\abs{\hat{f}(x)-f(x)}^2dx.
$ 
%is minimized.
%Here $c\in(0,1/2)$ is a small constant that constrains the evaluation of the mean-square error
%to a strict \emph{interior} of $[0,1]$, which avoid ``boundary effects'' of potential nonparametric estimators.
%To simplify notations, in the rest of this paper we shall use $\|\hat f-f\|_2^2$ and $\|\hat f-f\|_{2,[c,1-c]^d}^2$ interchangeably
%when no confusion can be caused.

%One important aspect of the nonparametric estimation problems is the minimal assumptions imposed on the unknown models of interest $f$. More specifically, only smoothness type assumptions such as bounded high-order derivatives of $f$ are imposed, which contrasts classical parametric approaches that formulate $f_0$ as linear or generalized linear models. 
%This allows the unknown function to belong to large and comprehensive function classes, which cover most functions that arise in practical applications.

The nonparametric regression problem has a long history of study, dating back to the 1920s~\citep{whittaker1922new}.
A large family of methods have been developed and their properties analyzed,
including kernel smoothing \citep{friedman2001elements,gyorfi2006distribution},
spline smoothing~\citep{reinsch1967smoothing,geer2000empirical,green1993nonparametric},
wavelet smoothing~\citep{donoho1998minimax,donoho1994ideal,hardle2012wavelets}
and local regression methods \citep{fan1992variable,fan1993local,fan1996local}.
We refer the readers to the excellent books of \citep{gyorfi2006distribution,tsybakov2009introduction,larry2006all}
for the comprehensive literature on the nonparametric regression problem.

%robustness issue
In many real world applications, the observations may subject to systematic or even adversarial noise.
Classically, the sensitivity of conventional statistical procedures to outliers was noted by~\citet{tukey1975mathematics} who observed that estimators like the empirical mean can be sensitive to even a single gross outlier.
In the nonparametric setting, similarly, the adversary noise can completely break these nonparametric estimators.

The formal study of robust estimation was initiated by~\citet{huber1964robust,huber1965robust},
who considered the $\epsilon$-contamination model where the observations are distributed from the mixture:
\[
y_{\vect{i}} \overset{i.i.d.}{\sim} (1-\epsilon) P(x_{\vect{i}}) + \epsilon Q(x_{\vect{i}}).
\]
Here $P\left(\vect{x}_\vect{i}\right) = N(f(x_{\vect{i}}),1)$, Gaussian distribution with mean $f(x_{\vect{i}})$ and variance $1$, $Q(x_{\vect{i}})$ is an arbitrarily adversary distribution and $\epsilon$ is the expected fraction of adversarial outliers.
\footnote{The exact Gaussianity of $P$ is not required and our analysis generalizes to other sub-Gaussian benign noise distributions as well.}
% while no assumptions are placed on the adversarial distribution $Q$.

The performance of robust estimators can be evaluated under a \emph{minimax} framework recently formulated in the statistics literature~\citep{chen2015robust,chen2016general,gao2017robust,liu2017density}. %have focused on providing a minimax perspective by characterizing both minimax upper and lower bounds on the performance of estimators under a variety of settings. 
Typically, the minimax rates under the Huber's $\epsilon$-contamination model can be divided into two components: 
(1) The \emph{contamination dependence}, which concerns the dependence of the estimation error on the contamination parameter $\epsilon$;
(2) The \emph{statistical rate}, which coincides with the classical notion of statistical consistency and goes to zero as we accumulate more data (i.e., $n\to\infty$).
We say a robust estimator is \emph{minimax optimal} if both its contamination dependence 
and statistical rate match their information-theoretical limits, up to absolute constants.
%Both contamination dependence and statistical rate may depend on the noise level and function class being considered (e.g. dimension in parametric setting and smoothness in the nonparametric setting).
%However, estimators proposed in these works are computationally intractable.

%nonparametric Regression

\paragraph{Our contributions:}
In this paper, we present a simple and easy-to-compute local binning median method which can effectively remove all the adversarial noise.
Theoretically, this method is statistically optimal in terms of both contamination dependency and statistical rate up to absolute constant for estimating functions in \holder class with smoothness parameter $\beta\in[0,1]$ (a formal definition of \holder function class is given in Sec.~\ref{subsec:holder}).

Furthermore, we propose a generic method for estimating functions with higher orders of smoothness.
Our method has two steps. 
First, we use a local binning median estimator to remove adversarial noise.
Next, we take a non-parametric procedure on top of these medians where the specific choice of the procedure depends on user's knowledge of the underlying function.
In particular, we show that if we use kernel smoothing or local polynomial regression, we can achieve minimaxity for \holder and Sobolev function classes with higher orders of smoothness.
We remark that both steps can be computed efficiently in polynomial time.

To our knowledge, these are the first computationally efficient algorithms for nonparametric regression under the Huber's $\epsilon$-contamination model with provably optimal contamination dependency and statistical rate.

\subsection{Related work}
\label{sec:related}

%classicalworks
Some of the classical literature on robust statistics focused on the design of robust estimators and the study of their statistical properties (see, for instance, the works of ~\citet{huber2011robust,hampel2011robust}).
The major drawback of many of these classical robust estimators is that they are either heuristic in nature
(for instance, methods based on Winsorization~\citep{hastings1947low} are generally not optimal in
the minimax sense), or are computationally intractable (for instance, methods based on Tukey’s depth~\citep{tukey1975mathematics} or on $\ell_1$ tournaments~\citep{yatracos1985rates}).

In the nonparametric setting, most papers focused on density estimation problems~\citep{acharya2017sample,chan2014near,diakonikolas2016efficient,daskalakis2012learning,liu2017density}.
On the other hand, existing works on robust nonparametric regression only focused on heavy-tailed noise~\citep{fan1994robust}, but not Huber's $\epsilon$-contamination model.
To our knowledge, only~\citet{gao2017robust} considered nonparametric regression problems in this model.
This estimator is based on the concept of \emph{regression depth}, a generalization of Tukey's depth and achieves optimal contamination dependency and statistical rate for Sobolev function class.
Unfortunately, there is no known polynomial time algorithm to compute such estimators.

Methodologically, our proposed algorithms also resemble the two-step method considered in \citep{brown2008robust}.
However, their analysis is based on asymptotic equivalence~\citep{cai2009asymptotic} and cannot be adapted to Huber's $\epsilon$-contamination model.
Instead, we propose a decoupled analysis and only use elementary concentration inequalities to prove the upper bound of this estimator.
Given the importance of Huber's $\epsilon$-contamination model, we believe our proof techniques can also be used in other robust statistical estimation problems.

%recent works from TCS
Recent works from the theoretical computer science community proposed new methods for robust statistical estimation with polynomial running time guarantees.
\citet{diakonikolas2016robust,lai2016agnostic,charikar2017learning} provide some of the first computationally tractable, provably robust estimators
with near-optimal contamination dependence in a variety of settings.
More recently, generalization to high-dimensional setting~\citep{balakrishnan2017computationally} is also studied.
However, these works only focused on the parametric models.

\subsection{Function Classes}\label{subsec:holder}
The \holder class is a popular choice of smoothness classes for nonparametric estimation problems.
The following definition gives a rigorous mathematical formulation of \holder classes considered in this paper:
\begin{defn}
\label{defn:holder}
For $f: \left[0,1\right]^d \rightarrow \mathbb{R}$, if $f\in \Lambda\left(\beta,L\right)$, then it satisfies for any $x,x' \in \left[0,1\right]^d$\begin{align*}
\sum_{j=0}^{\ell}\sum_{\alpha_1+\cdot+\alpha_d=j}\abs{f^{(\vect{\alpha},j)}(x)} + \sum_{\alpha_1+\cdots+\alpha_d=k}\frac{\abs{f^{(\vect{\alpha},k)}(x)-f^{(\vect{\alpha},k)}(x')}}{\norm{x-x'}_\infty^{\beta-\ell}} \le L
\end{align*}
where $\ell = \lfloor \beta \rfloor$ and $f^{(\vect{\alpha},j)(x)} = \partial^j f(x)/\partial x_1^{\alpha_1}\ldots\partial x_d^{\alpha_d}$.
\end{defn}
At a higher level, functions belonging to the \holder class have their derivatives and derivative changes \emph{uniformly} bounded on the unit interval $[0,1]^d$.

While \holder class imposes uniform smoothness condition, Sobolev class only assumes the averaged magnitudes of the derivatives are bounded.
Here we adopt the definition from~\cite{nemirovski2000topics}.
\begin{defn}
\label{defn:sobolev}
	For $f: \left[0,1\right]^d \rightarrow \mathbb{R}$, if $f\in \Sigma\left(\beta,p,L\right)$, it satisfies\begin{align*}
\norm{D^\beta f}_p \le L
	\end{align*}
	where $\beta \ge 1$, $p \ge d$ are integers and $D^\beta f (\cdot)$is the vector function comprised of all partial derivatives (in terms of distributions) of f with order $\beta$.
\end{defn}

%\begin{defn}
%For any smooth parameter $\beta>0$ and constant $L>0$, 
%let $\ell=\lfloor\beta\rfloor$ be the largest integer that lower bounds $\beta$.
%The \holder class $\Lambda(\beta,L)$ consists of all $\ell$-times differentiable functions $f:[0,1]\to\mathbb R$ that satisfy
%\begin{equation*}
%\max\left\{\max_{1\leq j\leq\ell}\|f^{(j)}\|_\infty, \sup_{x,y\in[0,1]}\frac{|f^{(\ell)}(x)-f^{(\ell)}(y)|}{|x-y|^{\beta-\ell}}\right\}
%\leq L.
%\end{equation*}
%\end{defn}

%We also need the following \holder class for multivariate function.

%\subsection{Organization}
%\label{sec:org}
%This paper is organized as follows.
%In Section~\ref{sec:pre}, we list necessary backgrounds for robust nonparametric regression.
%In Section~\ref{sec:naive}, we briefly discussed some naive approaches and why they fail.
%In Section~\ref{sec:regression_fixed}, we analyze the local  binning median estimator and show it achieves the optimal contamination dependency and statistical rate for \holder class with smoothness parameter $0\le \beta \le 1$ up to absolute constant.
%In Section~\ref{sec:higher_smoothness}, we show how to use this local binning median procedure as a preprocessing step which, combined with kernel smoothing method, achieves the near optimal error rate up for \holder class with smoothness parameter $\beta > 1$.
%We use simulations to verify our theories in Section~\ref{sec:exp} and conclude in Section~\ref{sec:con}.
%We put some of our technical proofs in the appendix.

%\subsection{Related Works}
%\label{sec:rel}
%\input{rel.tex}

\section{Some Natural Approaches and Their Problems}
\label{sec:naive}
In this section we discuss some natural approaches and their problems.
For the ease of presentation we focus on the one dimensional setting.

\paragraph{Example I: Direct Kernel Smoothing}
We first consider directly applying kernel smoothing estimator to this problem.
Let $x_0$ be a query point, recall the kernel smoothing method has the following weighted sum form :$
\hat{f}(x_0) = \sum_{i=1}^{n}\frac{1}{h}K_i^h(x_0)y_i.
$ where $K_i^h(\cdot)$ is a kernel function and $h$ is the band-width.
Consider the case that there exists some $i$ such that $K_i^h(x_0) > 0$ and $\xi_i \sim Q$, the noise is sampled from the adversarial distribution.
Because $Q$ is arbitrary, $\xi_i$ can be infinity and in turn $y_i$ is infinity.
Therefore, in this situation our estimator will output infinity which is undesirable.
One may think this event happens with low probability.
However, note for every observation $y_i$ we have $\epsilon$ probability that $\xi_i \sim Q$.
Thus it is not hard to show with high probability that there exists one $i$ such $K_i^h(x_0) > 0$ and $y_i = \infty$.

\paragraph{Example II: Truncated Kernel Smoothing}
A natural to deal with the extremely large adversarial noise is to use truncation.
For \holder class, since we know $\norm{f}_\infty \le L$ and $P$ is standard Gaussian distribution.
We can modify the kernel smoothing estimator as $
\hat{f}(x_0) = \sum_{i=1}^{n}\frac{1}{h}K_i^h(x_0)\bar{y}_i
$ where \begin{align*}
	\bar{y}_i = \begin{cases}
	L+c \text{ if } y_i \ge L + c\\
	y_i \text{ if } -(L+c\sigma) \le  y_i \le L + c\\
	-(L+c) \text{ if } y_i \le -(L+c)
	\end{cases}.
\end{align*}
Here we choose $c > 1$ to ensure most samples whose noise are from $P$ are not being truncated~\citep{diakonikolas2016robust}.
This estimator rules out the infinite issue.
However, consider the following scenario.
Let $x_0$ be a query point and $f(x_0)=0$.
Note around $x_0$ there are approximately $\epsilon$ adversarial observations and if we set the adversarial distribution to be a unit mass on the infinity, then there are approximately $\epsilon$ fraction of $\bar{y}_i$s are $L+c$.
These points will incur an $\Omega\left(\epsilon^2 (L+1)^2\right)$ squared bias in our estimation.
In Section~\ref{sec:lower_bound}, we will show the optimal lower bound for the contamination dependency is $\Omega\left(\epsilon^2\right)$ and it  does not depend on $L$.
When $L$ is relatively large, this estimation will be far off from the truth.
We also verify this phenomenon by simulation in Section~\ref{sec:exp}.

\section{Local Binning Median Estimator}
\label{sec:lbm}
We introduce the local binning estimator and its theoretical properties in this section.
%We introduce our main technical results for one dimensional non-parametric regression in this section,
%which establish minimax rates of convergence for robust estimation over \holder function classes under the Huber's $\epsilon$-contamination model.
%We first describe a local binning median estimator and derive its minimax convergence rates over the \holder class with smoothness parameter upper bounded by 1.
%Afterwards, a kernel smoothing post-processing step is considered which extends the applicability of the robust estimator to smoother \holder functions.
%Finally, information-theoretical lower bounds are proved which shows the minimax optimality of the proposed robust estimators over the \holder class. 
The interval $\left[0,1\right]$ is evenly divided into $m$ sub-intervals.
The bin $\vect{j} =\left(j_1,\ldots,j_d\right)$ corresponds to the $d$-dimensional box $[\frac{j_1-1}{m}, \frac{j_1}{m}) \times \cdots \times [\frac{j_d-1}{m}, \frac{j_d}{m}) $ and has $s = \frac{n}{m^d}$ design points.
\footnote{For simplicity of presentation, we assume $m$ and $s$ are both integers. 
It is straightforward to generalize to the case where $n$ is not a multiple of $s$.}
For a query point $x_0\in[0,1]^d$, let $\vect{j}\in[m]$ be the label of the unique bin that contains $x_0$.
The local binning median estimator evaluated at $x_0$ is then defined as
\[
\hat{f}\left(x_0\right) := \median\left\{y_\vect{i}\right\}_{\vect{i} \in \text{ bin } \vect{j}}.
\] where $\vect{i} = \left(i_1,\ldots,i_d\right)$ is the index.
It is clear that $\hat f(x_0)$ is easily computed in time $O(s)$, as the median can be computed in linear time (see, e.g., \citet{kleinberg2006algorithm}).

Before delving into detailed statistical properties $\hat f$, we provide motivations for both the median operator and the local binning approach taken in $\hat f$.
\paragraph{\emph{The median}.}
Median is widely used in robust estimation problems for dealing with adversarial noise~\citep{huber2011robust}, heavy tail noise~\citep{brown2008robust}, etc.
There are two fundamentally useful properties about median.
First, taking the median ensures that, as long as the majority of the samples are good, the output will not be affected by the small amount of outliers.
In addition, the median itself has an concentration effect similar to what the mean does.
Our proof heavily relies on these two properties. %(c.f. Theorem~\ref{thm:fixed_design_median}).

\paragraph{\emph{The local binning}.}
Binning is used to exploit the smoothness of the function class.
Since the binning estimator only uses observations near the query point, for smooth function class, this estimator will incur only a small bias.
This is the same idea in kernel smoothing where we often put more weights on points that are near the query point.

We are ready ready to state our main lemma which characterizes the performance of this local binning median estimator.
\begin{lem}\label{lem:meta_lemma_uniform_mult}
Suppose $\epsilon \le 1/4$.
Let $s = \frac{n}{m^d}$ and $z_\vect{j} = \median\left\{y_\vect{i}\right\}_{\vect{i} \in \text{ bin } \vect{j}}$.
With probability at least 0.9, for any $\vect{j} \in [m]^d$, we have \begin{align*}
		z_\vect{j} = f\left(\frac{\vect{j}}{m}\right) + \triangle_j + \eta_j
	\end{align*} where $\abs{\triangle_\vect{j}} \le \max_{x \in \text{ bin }\vect{j}}\abs{f\left(x\right)-f\left(\frac{\vect{j}}{m}\right)}$, $\abs{\expect{\eta_\vect{j}}} \le C\left(\epsilon + \frac{\log m}{s}\right)$ and $\variance\left(\eta_\vect{j}\right) \le \frac{C}{s}$.
\end{lem}

Lemma~\ref{lem:meta_lemma_uniform_mult} shows a form of bias-variance trade-off.
$\abs{\triangle_{\vect{j}}}$ characterizes the bias incurred by our binning step.
This term is small if the function we are estimating has smoothness property.
$\abs{\expect{\eta_{\vect{j}}}}$ characterizes the biased incurred by the adversary and stochastic noise.
Later in Section~\ref{sec:lower_bound}, we will show that $\epsilon$ is unavoidable in Huber's $\epsilon$-contamination model.
The $\frac{\log m}{s}$ is small if $s$ is relatively large.
Lastly, $\variance\left(\eta_\vect{j}\right)$ is the variance incurred by the stochastic noise.
Note we have $\variance\left(\eta_\vect{j}\right) = O\left(\frac{1}{s}\right)$ because the variance of the median of $s$ sub-Gaussian random variable is $O\left(\frac{1}{s}\right)$.

With this meta Lemma at hand, we can direct derive upper bounds for some particular function classes.

\begin{thm}[Estimation Error Bound for \holder Function Class]\label{thm:l2_upper_bound}
Suppose $f\in\Lambda(\beta,L)$ for some $\beta\in(0,1]$ and $m\asymp n^{\frac{1}{2\beta+d}}L^{\frac{2}{2\beta+d}}$.
Then there exists an absolute constant $C>0$ such that with probability at least 0.9, 
\begin{align*}
\norm{\hat{f}-f}_2^2 \le C\left(L^{\frac{2d}{2\beta+d}}n^{-\frac{2\beta}{2\beta+d}}+\epsilon^2\right).
	\end{align*} 
\end{thm}

%\simon{To be checked later}
%\begin{thm}[Estimation Error Bound for Sobolev Function Class]\label{thm:l2_upper_bound_sob}
%Suppose $f\in\Lambda(\beta,L)$ for some $\beta\in(0,1]$ and $m\asymp n^{\frac{1}{1+2\beta}}L^{\frac{2}{2\beta+d}}$.
%Then there exists an absolute constant $C>0$ such that with probability at least 0.9, 
%\begin{align*}
%\norm{\hat{f}-f}_2^2 \le C\left(L^{\frac{2}{2\beta+1}}n^{-\frac{2\beta}{2\beta+d}}+\epsilon^2\right).
%	\end{align*} 
%\end{thm}

\begin{rem}
Though Theorem \ref{thm:l2_upper_bound} seems to appear only to \holder class with $\beta\leq 1$, one should note that 
$\Lambda(\beta,L)\subseteq\Lambda(\beta',L)$ always holds for any $\beta\geq\beta'$.
Therefore, in cases where $\beta>1$, the scaling of $m$ and the upper bound on $\|\hat f-f\|_2^2$
remain valid with $\beta$ replaced by $\beta'=\min\{\beta,1\}$.
\end{rem}

\begin{rem}
Theorem \ref{thm:l2_upper_bound} suggests two potential sources of error in $\|\hat f-f\|_2^2$.
The first term $L^{2d/(2\beta+d)}n^{-2\beta/(2\beta+d)}$ is the classical statistical rate in nonparametric regression and goes to 0 as $n$ increases to infinity.
The second term $\epsilon^2$ represents the contamination dependency and typically remains constant under large-$n$ asymptotic regimes.
Importantly, this term does \emph{not} depend on the \holder smooth constant $L$.
%Additionally, this term does not depend a function $f$. %in sharp contrast the naive truncated kernel smoothing estimator we discussed in the last section.
\end{rem}

%\begin{rem}
%Here we use median as an illustrative example. It is possible to use other local robust mean estimator to further exploit the function class properties.
%\end{rem}

\section{Post-processing to Exploit Higher Smoothness}
\label{sec:postprocessing}
While the local binning median estimator $\hat f$ is intuitive, the number of bins $m\asymp n^{\frac{1}{1+2\beta'}}L^{\frac{2}{2\beta'+1}}$ for $\beta'=\min\{\beta,1\}$
is likely to be too large for smoother \holder functions where $\beta>1$ or Sobolev functions.
This leads to under-smoothing and loss of statistical efficiency.
The main reason is a simple local median step cannot exploit the smoothness in the derivatives of the underlying function (c.f. Definition~\ref{defn:holder} and Definition~\ref{defn:sobolev}).

In this section we show a simple post-processing step can exploit this higher smoothness property.
Our main observation is from Lemma~\ref{lem:meta_lemma_uniform_mult}.
Recall we have with high probability
\begin{align*}
		z_\vect{j} = f\left(\frac{\vect{j}}{m}\right) + \triangle_\vect{j} + \eta_\vect{j}.
\end{align*}
We view $z_{\vect{j}}$ as a noisy observation of $f\left(\frac{\vect{j}}{m}\right)$ where the noise is $(\triangle_\vect{j} + \eta_\vect{j})$ which has bounded bias $O(\epsilon+\frac{\log m}{s})$ and variance $\frac{1}{s}$.
Therefore we can take any algorithm that works for non-parametric problem under stochastic noise with bounded bias.
Formally, we define a non-parametric algorithm acting on $n$ samples as a map\begin{align}
\mathcal{A}: 
\mathcal{A}\left(\left\{y_i\right\}_{i=1}^n, x_0\right) \rightarrow \check{f}\left(x_0\right) \label{eqn:nonparam_algo}
\end{align}
where $x_0$ is a query point and $\left\{y_i\right\}_{i=1}^n$ are observations.
The final performance will depend on how the noise affect the algorithm.
This abstract procedure is summarized in Algorithm~\ref{algo:post_processing}

\begin{algorithm}[tb]
	\caption{The local binning median estimator with  post-processing}
	\label{algo:post_processing}
	\begin{algorithmic}[1]
		\STATE \textbf{Input}: $\left\{y_{\vect{i}}\right\}_{i=1}^n$, number of bins $m$, $s=\lfloor n/m^d\rfloor$, a nonparametric algorithm $\mathcal{A}$ defined in Equation~\eqref{eqn:nonparam_algo}, query point $x_0$.
		\STATE \textbf{Output}: estimated function value $\check{f}(x_0)$.
		\FOR{$\vect{j} = [m]^d$}
		\STATE  Compute $z_\vect{j} = \median \left\{y_\vect{i}\right\}_{i \in \text{ bin }\vect{j}}$.
		\ENDFOR    
		\STATE Compute $\check{f}\left(x_0\right) = \mathcal{A}\left(\left\{\vect{z}_\vect{j}\right\}_{\vect{j} \in [m]^d},x_0\right)$
	\end{algorithmic}
\end{algorithm}

Now we use two specific non-parametric estimators to illustrate this idea.

\subsection{Post-processing by kernel smoothing}
\label{sec:postprocessing_kernl}
We first choose the non-parametric estimation $\mathcal{A}$ to be kernel smoothing for one dimensional robust non-parametric estimation to illustrate the usage of Algorithm~\ref{algo:post_processing}.
This estimator which significantly improves the statistical efficiency of $\hat f$ for highly smooth functions in the \holder class.
Let $K:\mathbb R\to\mathbb R^*$ be a kernel function selected by the data analyzer.
We assume the kernel function $K(\cdot)$ satisfies the following properties:
\begin{enumerate}
\item $K$ is supported on $[-1,1]$, and $\int_{-1}^{1}K(x)dx = 1$; 
\item $\int_{-1}^{1} K(x) x^k dx = 0$ for all $k=1,\ldots,\ell$;
\item $\int_{-1}^{1} K^2(x) dx =: \kappa < \infty$.
\end{enumerate}
We remark that the above assumptions on properties of $K(\cdot)$ are not restrictive,
as the choice of the kernel function $K(\cdot)$ belongs to the data analyzer and there are many simple functions
satisfying (A1) through (A3) for various values of $\ell$.
Examples include the box kernel $K(u)=\mathbb I[|u|\leq 1]$, the triangular kernel $K(u)=\max\{0,1-|u|\}$
and the Epanechnikov kernel $K(u)=\max\{0, \frac{3}{4}(1-u^2)\}$.

Given $0 < h < 1/2$ and a query point $x \in \left(h,1-h\right)$ we define for $j=1,\ldots, m$,
\begin{align}
K_j^h(x) = \int_{(j-1)/m}^{j/m}\frac{1}{h}K\left(\frac{x-u}{h}\right) du. \label{eqn:kernel_def}
\end{align}
The kernel smoothing step can then be understood as applying the $K_j^h(\cdot)$ operator over all observations or their local binning median surrogates and aggregating the results.

The next theorem states the main result for the combined two-step procedure, 
which shows that $\check f$ enjoys improved error convergence for highly smoother functions $f\in\Lambda(\beta,L)$ with $\beta>1$.
\begin{thm}\label{thm:median_kernel}
	Suppose $f\in\Lambda(\beta,L)$ for $\beta\in(1,\infty)$ and let $c \in (0,1/2)$ be a small constant.
If in Algorithm~\ref{algo:post_processing}, we choose $\mathcal{A}$ to be the kernel smoothing estimator defined in Equation~\eqref{eqn:kernel_def}, 	
$m\asymp \sqrt{n}/{\sqrt[4]{\log n}} $ and $h\asymp \left(nL^2\right)^{-\frac{1}{2\beta+1}}$,
then there exists an absolute constant $C>0$ and a constant $C_L>0$ depending on $L$ such that with probability at least 0.9, 
	\begin{align*}
		\int_{c}^{1-c}\abs{\check f(x)-f(x)}^2 dx \le C\left(C_Ln^{-\frac{2\beta}{2\beta+1}}+\epsilon^2\right).
	\end{align*} %where $C_L$ is a constant that only depends on $L$.
\end{thm}

\begin{rem}
Theorem \ref{thm:median_kernel} nicely complements the results in Theorem \ref{thm:l2_upper_bound}
by applying to ``smoother'' functions $f\in\Lambda(\beta,M)$ with $\beta>1$.
The error bound is very similar to the one in Theorem \ref{thm:l2_upper_bound},
except that the constant in front of the statistical rate term $n^{-2\beta/(2\beta+1)}$ is more involved and we decided to suppress its exact dependency on $L$.
The contamination dependency term $\epsilon^2$ is again \emph{independent} of smoothness constant $L$.
\end{rem}

\begin{rem}
We consider the integrated loss over the interval $[c,1-c]$ to
constrains the evaluation of the mean-square error to a strict interior of $(0,1)$ which avoid “boundary effects” of the kernel smoothing estimator.
\end{rem}

\begin{rem}
It possible to extend kernel smoothing estimator to $d > 1$.
However, it requires more complicated kernel construction, especially for the fixed design.
Therefore, for $d>1$ we consider another estimator, local polynomial regression for estimation.
\end{rem}

\subsection{Post-processing by Local Polynomial Regression}
\label{sec:mult_dim}
%In this section we generalize the idea in one-dimension to multivariate setting.
%Let $d$ be the dimension and let the design points be $x_{\vect{i}} = \left(x_{i_1},\ldots,x_{i_d}\right)$ where $i_k \in [n^{1/d}]$ and observe \[
%y_{\vect{i}} \sim \left(1-\epsilon\right)\mathcal{N}\left(f\left(\frac{\vect{i}}{n^{1/d}}\right),1\right) + \epsilon Q.
%\] 
%We use the local binning strategy again. 
%Suppose in total we have $m^d$ bins and we use bin $b_{\vect{j}}$ with $\vect{j}=\left(j_1,\ldots,j_d\right)$ to denote the cube $[\frac{j_1-1}{m},\frac{j}{m}) \times \cdots \times [\frac{j_1-1}{m},\frac{j}{m})$.
%Define $z_\vect{j}= \median\left\{y_{\vect{i}}\right\}_{\vect{x_i} \in b_\vect{j}}$.
%We have the following meta lemma whose is proof is identical to that of Lemma~\ref{lem:meta_lemma_uniform}.

In this section we choose $\mathcal{A}$ to be a local polynomial regression estimator for multivariate robust nonparametric regression.
We let $h$ be a fixed bandwidth parameter.
Given a query point $x \in \left[0,1\right]^d$, local polynomial regression uses the following least square solution as an estimate of $f$ around $x$\begin{align}
\hat{f}_h =\arg\min_{g \in \mathcal{P}_\ell}\sum_{\vect{j} \in [m]^d} \mathbb{I}\left[\frac{\vect{j}}{m} \in B_h^\infty\left(x\right)\right]\left(z_{\vect{j}}-g(\frac{\vect{j}}{m})\right)^2 \label{eqn:lpr}
\end{align}
where $\mathcal{P}_\ell$ denotes the set of polynomials with degree $d$, $B_h^{\infty}\left(x\right)$ is the a ball with radius $h$ in infinity norn and $\mathbb{I}\left[\cdot\right]$ is the indicator function.

The following theorems show the performance for \holder and Sobolev class.
\begin{thm}\label{thm:multivariate_lpr}
Suppose $f\in\Lambda(\beta,L)$ for $\beta\in(1,\infty)$ or $f\in\Sigma(\beta,p,L)$ for some positive integer $\beta,p$ that satisfy $\frac{\beta-1}{d} \ge \frac{1}{p}$.
If in Algorithm~\ref{algo:post_processing}, we choose $\mathcal{A}$ to be the local polynomial regression estimator defined in Equation~\eqref{eqn:lpr}, $m\asymp \sqrt{n}/{\sqrt[4]{\log n}} $ and $h\asymp \left(nL^2\right)^{-\frac{1}{2\beta+d}}$
then there exists an absolute constant $C>0$ and a constant $C_L>0$ depending on $L$ such that with probability at least 0.9, 
\begin{align*}
	\norm{\hat{f}-f}_2^2 \le C\left(C_Ln^{-\frac{2\beta}{2\beta+d}}+\epsilon^2\right).
\end{align*} 
\end{thm}

%\begin{thm}\label{thm:multivariate_lpr_sob}
%Suppose $f\in\Sigma(\beta,p,L)$ with $\frac{\beta-1}{d} \ge \frac{1}{p}$.
%If in Algorithm~\ref{algo:post_processing}, we choose $\mathcal{A}$ to be the local polynomial regression estimator defined in Equation~\eqref{eqn:lpr}, $m\asymp \sqrt{n}/{\sqrt[4]{\log n}} $ and $h\asymp \left(nL^2\right)^{-\frac{1}{2\beta+d}}$
%then there exists an absolute constant $C>0$ and a constant $C_L>0$ depending on $L$ such that with probability at least 0.9, 
%\begin{align*}
%	\norm{\hat{f}-f}_2^2 \le C\left(C_Ln^{-\frac{2\beta}{2\beta+d}}+\epsilon^2\right).
%\end{align*} 
%\end{thm}

Again this matches the minimax lower bound.
The proof is similar to that of Theorem~\ref{thm:median_kernel} where we just need to combine Lemma~\ref{lem:meta_lemma_uniform_mult} and standard analysis of local polynomial regression.

\section{Minimax lower bounds}
\label{sec:lower_bound}
We complement our positive results
with the following negative result,
which establishes minimax lower bounds for nonparametric estimation over the \holder and Sobolev function class
under the Huber's $\epsilon$-contamination model.

\begin{thm} %[Minimax Lower Bound for \holder Function Class in Huber's $\epsilon$-contamination Model]
	\label{thm:lower_bound_for_holder}
	For any $\beta,L>0$ and $\epsilon\in(0,1)$, there exists an absolute constant $C>0$ such that
	\begin{align*}
	\inf_{\hat{f}}\sup_{f \in \Lambda\left(\beta,L\right)}\Pr_{f,\epsilon,Q}\left[\|\hat{f}-f\|_2^2\ \geq C\left(\wp(L) \cdot n^{-\frac{2\beta}{2\beta+d}}+\epsilon^2\right)\right] \geq 1/4,
	\end{align*} 
	where $\wp(L)$ is a non-zero polynomial function of $L$.
\end{thm}

\begin{thm} %[Minimax Lower Bound for \holder Function Class in Huber's $\epsilon$-contamination Model]
	\label{thm:lower_bound_for_sobolev}
	Let $\beta,p$ be positive integers and satisfy $\frac{\beta-1}{d} \ge \frac{1}{p}$.
	For any  $\epsilon\in(0,1)$, there exists an absolute constant $C>0$ such that
	\begin{align*}
	\inf_{\hat{f}}\sup_{f \in \Sigma\left(\beta,p,L\right)}\Pr_{f,\epsilon,Q}\left[\|\hat{f}-f\|_2^2\ \geq C\left(\wp(L) \cdot n^{-\frac{2\beta}{2\beta+d}}+\epsilon^2\right)\right] \geq 1/4,
	\end{align*} 
	where $\wp(L)$ is a non-zero polynomial function of $L$.
\end{thm}

\begin{rem}
The lower bound can be again divided into two components, with the ``statistical rate'' component $\wp(L) n^{-2\beta/(2\beta+d)}$
going to zero as $n$ increases to infinity,
and the ``contamination dependency'' term $\epsilon^2$ that remains constant and is \emph{independent} of the smoothness constant $L$.
%$\poly\left(L\right)n^{-\frac{2\beta}{2\beta+1}}$ is the classical statistical rate for nonparametric regression.
%$\epsilon^2$ represents the contamination dependency.
%Notably, the contamination dependency does NOT depend on $L$, Lipschitz constant of function class. 
%Or equivalently, it only depends on the noise level.
\end{rem}

\begin{rem}
The minimax lower bound established in Theorem \ref{thm:lower_bound_for_holder} and~\ref{thm:lower_bound_for_sobolev}
matches the upper bounds in Theorems \ref{thm:l2_upper_bound}, \ref{thm:median_kernel} and~\ref{thm:multivariate_lpr}
for both the statistical rate component and the contamination dependency component, up to absolute constants and polynomial functions of $L$.
In particular, the contamination dependency is $\epsilon^2$ in all theorems without any dependency on the smoothness constant $L$,
which demonstrates an interesting ``de-coupling'' effect between benign and adversarial noise variables.
\end{rem}

\section{Numerical results}
\label{sec:exp}
\begin{figure*}[t!]
	\centering
	\begin{subfigure}[t]{0.29\textwidth}
		\includegraphics[width=\textwidth]{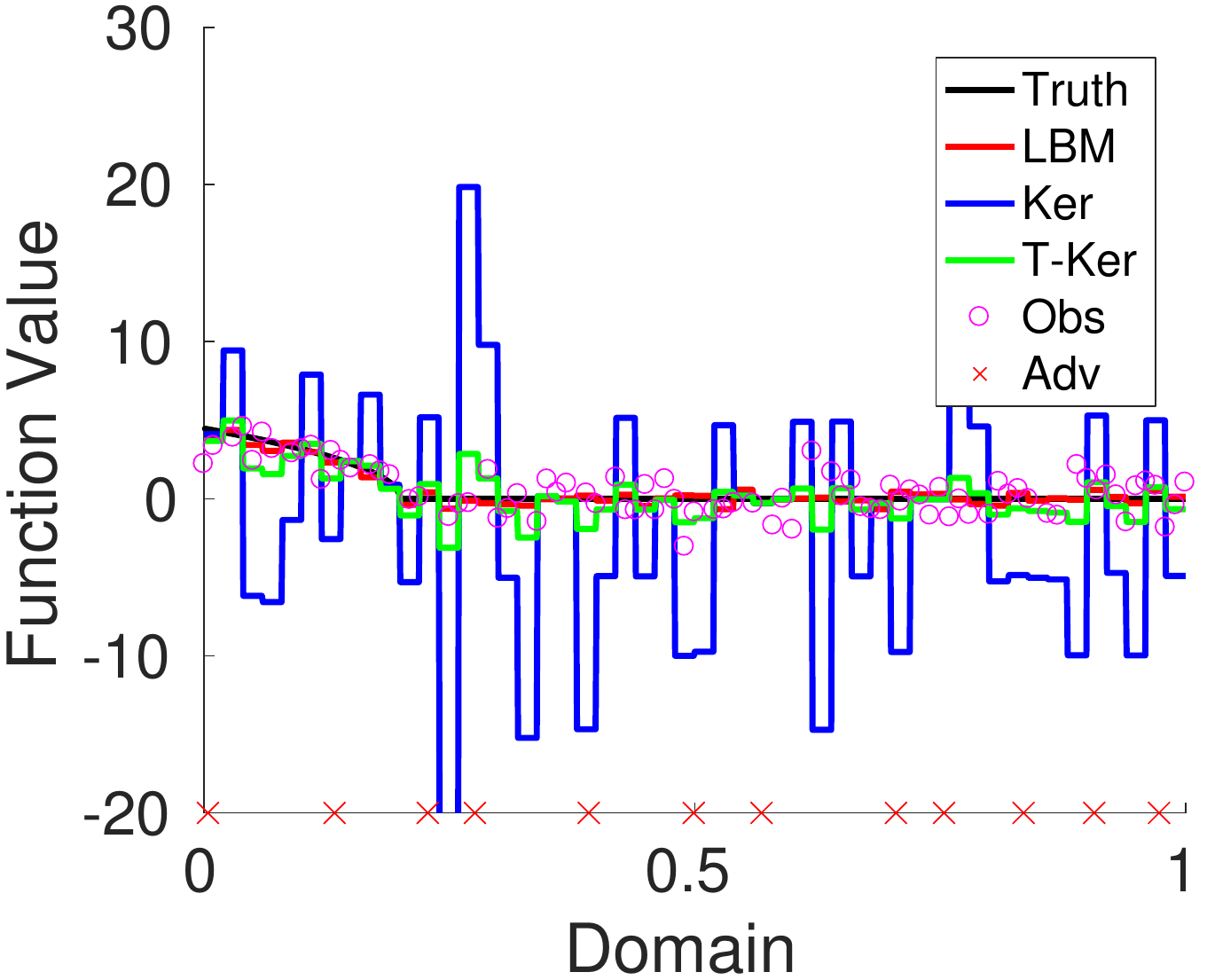}
		\caption{$L=10$}
	\end{subfigure}	
	\quad
	\begin{subfigure}[t]{0.29\textwidth}
		\includegraphics[width=\textwidth]{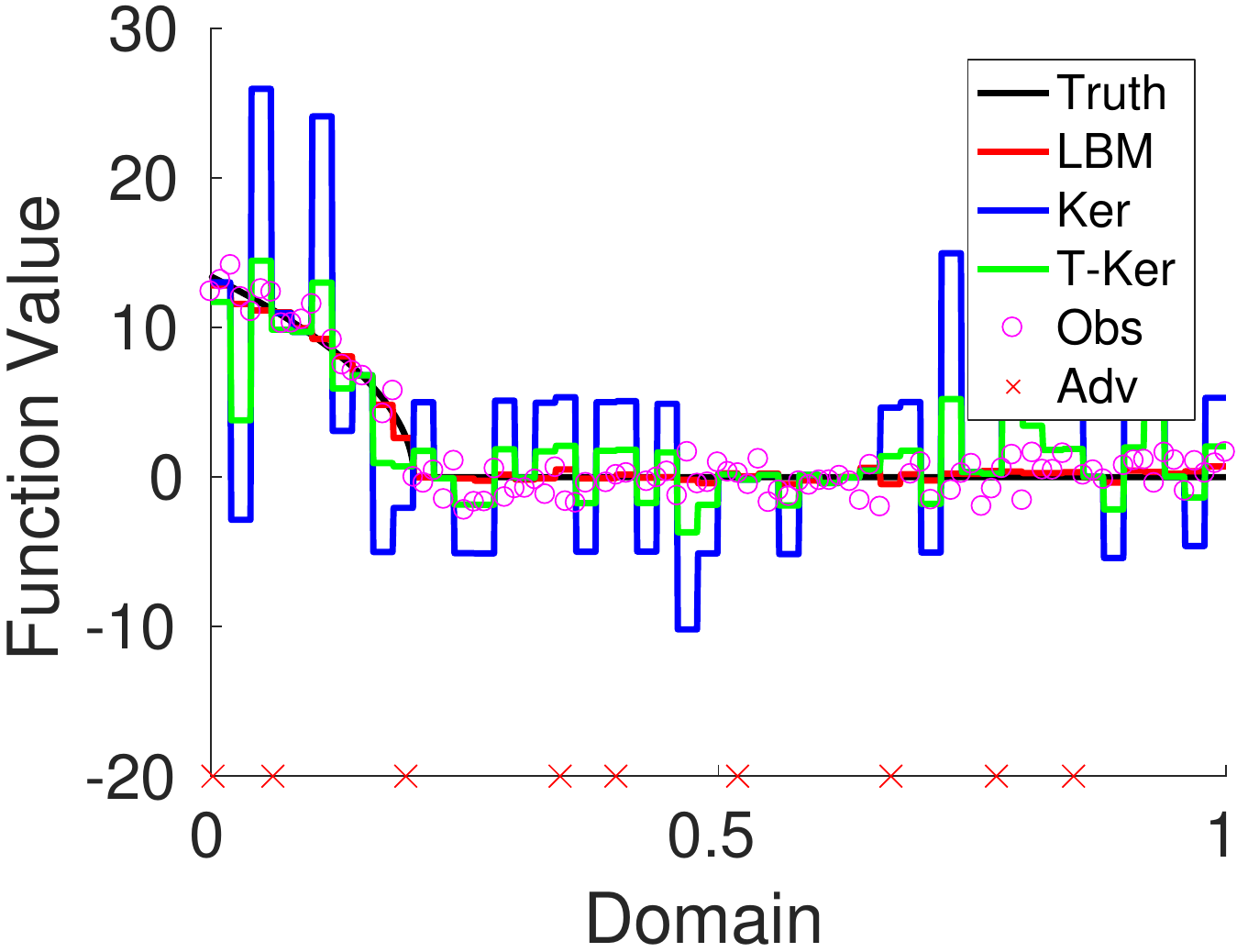}
		\caption{$L=30$}
	\end{subfigure}
	\quad
	\begin{subfigure}[t]{0.29\textwidth}
		\includegraphics[width=\textwidth]{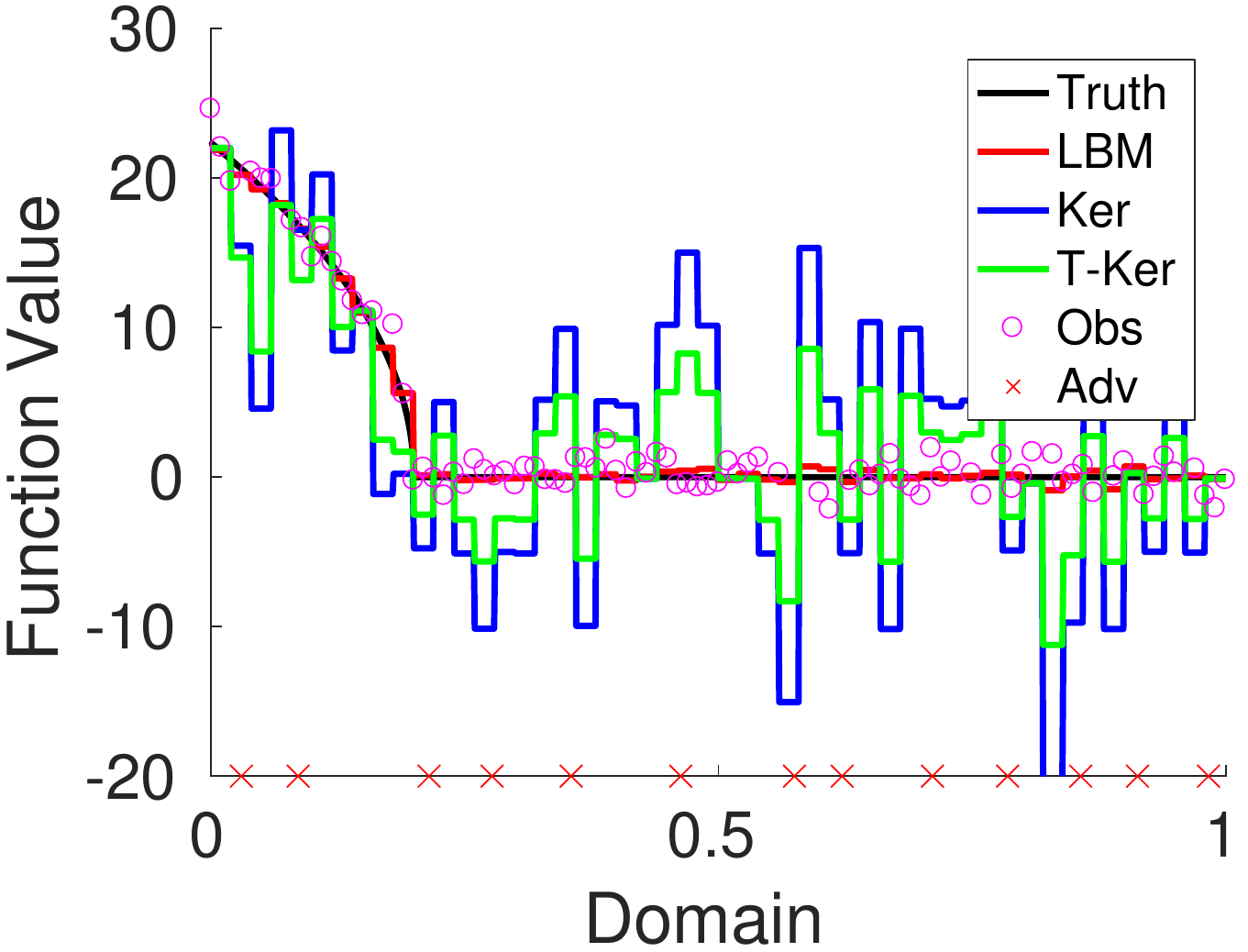}
		\caption{$L=50$}
	\end{subfigure}
	\caption{Experiments on estimating a function with low smoothness We choose $\beta = 0.5$, $\epsilon=0.1$, $\rho = 0.2$ and $f(x) = L(x-\rho)^\beta$  for $x \le \rho$ and $f(x)=0$ otherwise. Red 'X's represent the positions of adversarial points.
	}
	\label{fig:d1beta05}
%		\vspace{-0.3cm}
\end{figure*}

\begin{figure*}[t!]
	\centering
	\begin{subfigure}[t]{0.29\textwidth}
		\includegraphics[width=\textwidth]{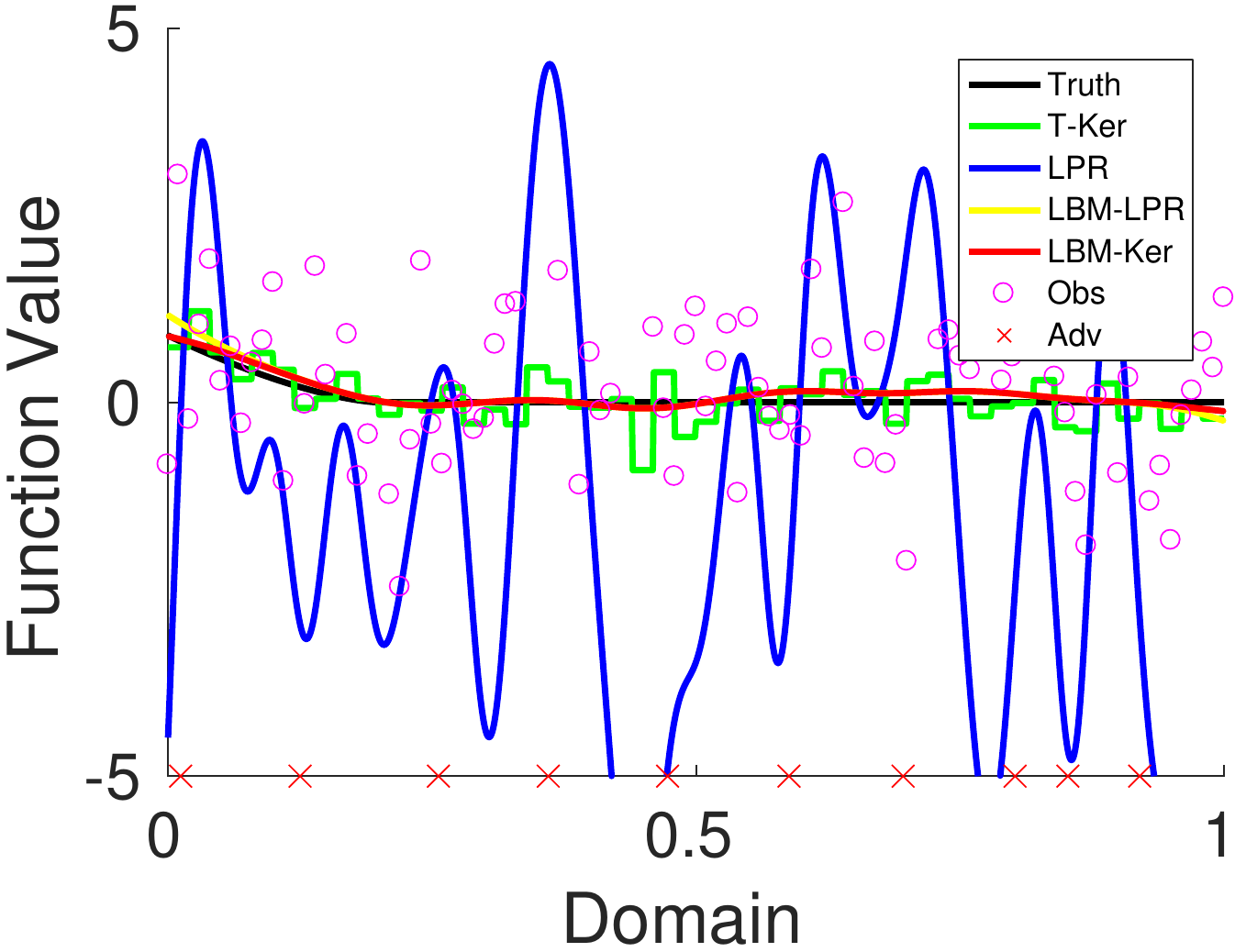}
		\caption{$L=10$}
	\end{subfigure}	
	\quad
	\begin{subfigure}[t]{0.29\textwidth}
		\includegraphics[width=\textwidth]{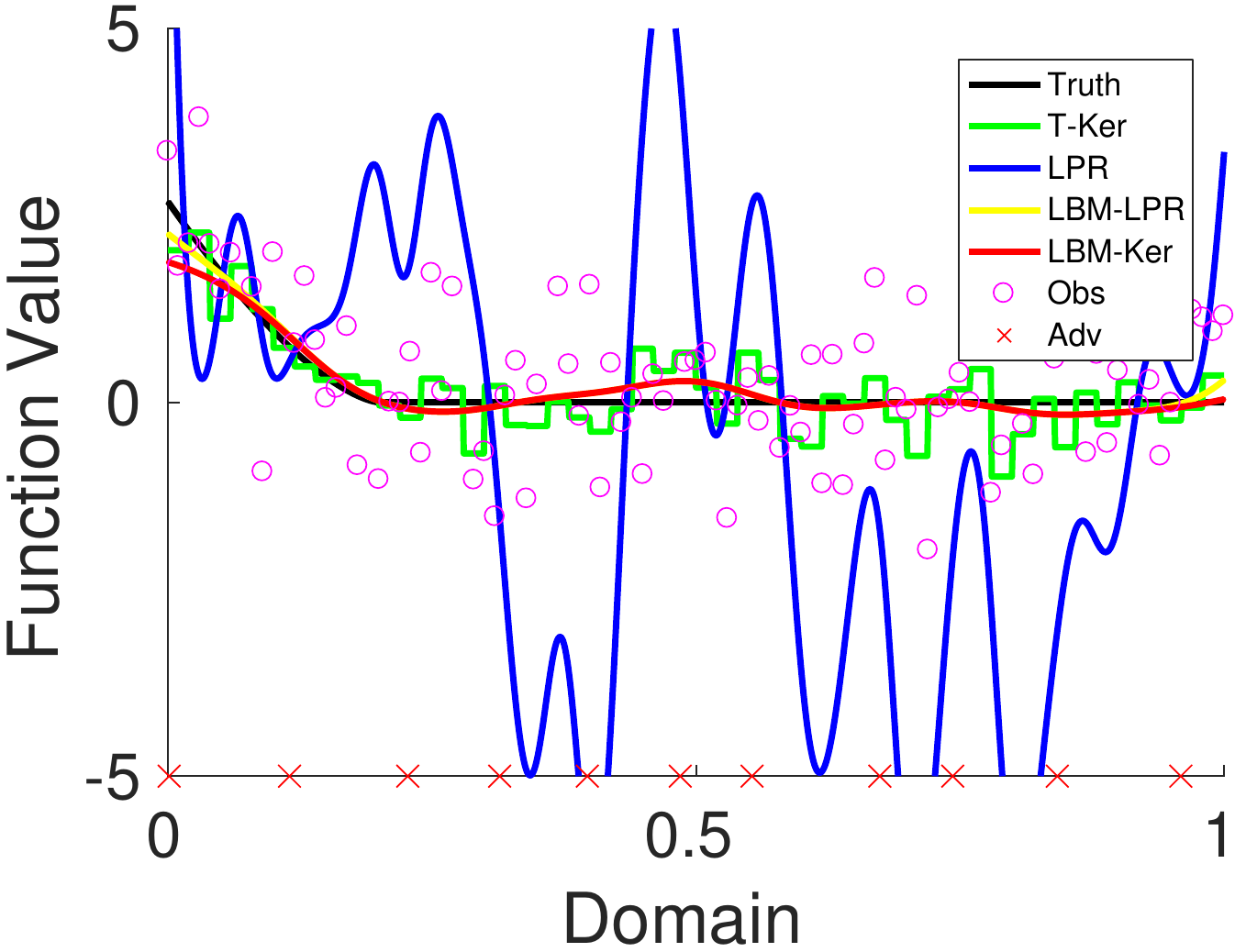}
		\caption{$L=30$}
	\end{subfigure}
	\quad
	\begin{subfigure}[t]{0.29\textwidth}
		\includegraphics[width=\textwidth]{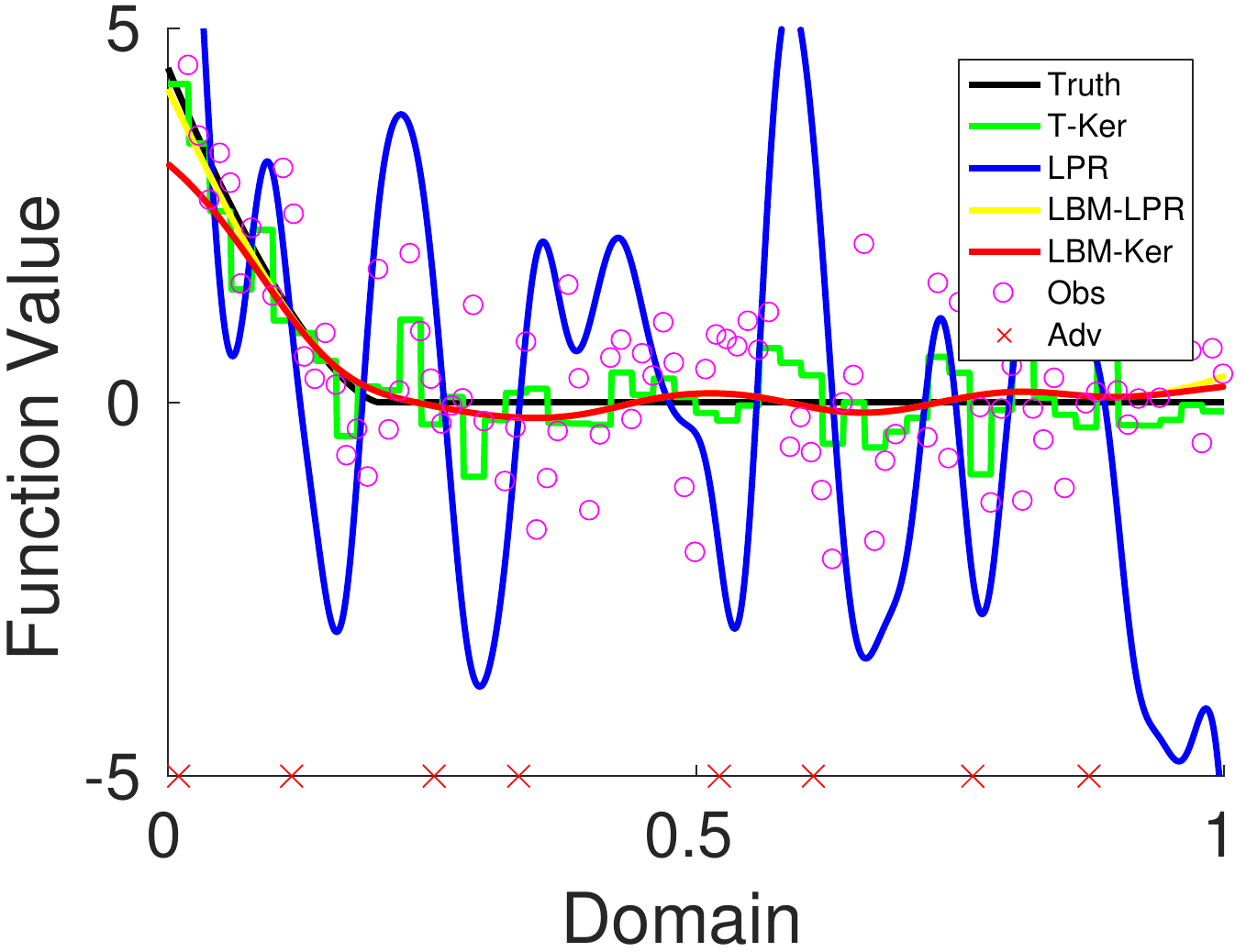}
		\caption{$L=50$}
	\end{subfigure}
	\caption{Experiments on estimating a function with high smoothness. We choose $\beta = 1.5$, $\epsilon=0.1$, $\rho = 0.2$ and $f(x) = L(x-\rho)^\beta$  for $x \le \rho$ and $f(x)=0$ otherwise. Red 'X's represent the positions of adversarial points.
	}
	\label{fig:d1beta15}

\end{figure*}

\begin{figure*}[t!]
	\centering
	\begin{subfigure}[t]{0.29\textwidth}
		\includegraphics[width=\textwidth]{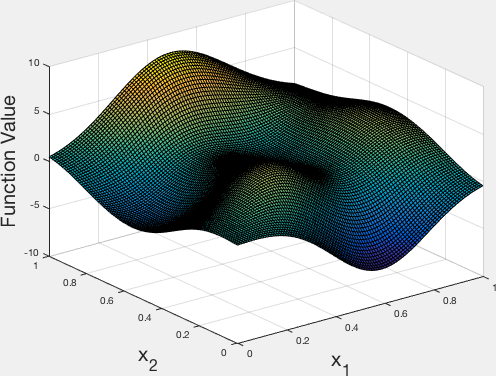}
		\caption{Original Function.}
	\end{subfigure}	
	\quad
	\begin{subfigure}[t]{0.29\textwidth}
		\includegraphics[width=\textwidth]{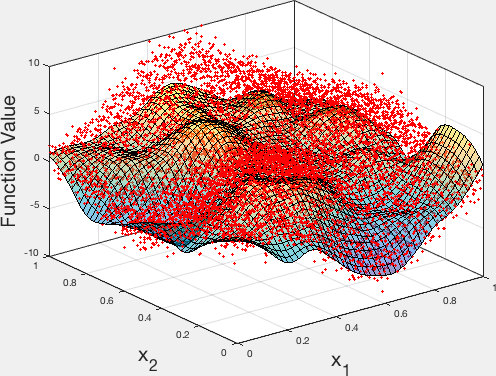}
		\caption{LPR estimator.}
	\end{subfigure}
	\quad
	\begin{subfigure}[t]{0.29\textwidth}
		\includegraphics[width=\textwidth]{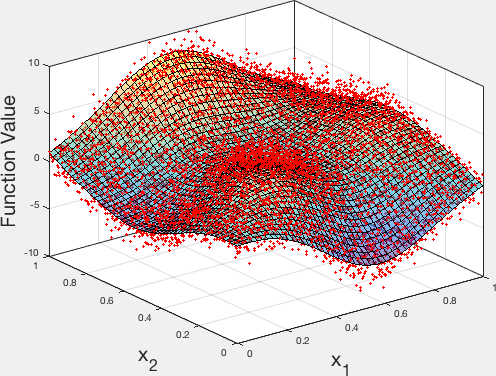}
		\caption{LBM+LPR estimator.}
	\end{subfigure}
	\caption{$2$-dimensional nonparametric estimation for peak function.
	}
	\label{fig:d2}
\end{figure*}

We first use simulations to verify our theoretical results.
In Figure~\ref{fig:d1beta05}-Figure~\ref{fig:d2}, we consider the following estimators:
%\begin{itemize}
%	\item \textbf{Kernel}: the classical kernel smoothing estimator.
%	\item \textbf{T-kernel}: truncated kernel smoothing estimator described in Section~\ref{sec:naive}. 
%	We use an additional hyperparameter $T$ to control the truncation level.
%	\item \textbf{LBM}: local binning median estimator described in Section~\ref{sec:lbm}. 
%	\item \textbf{LBM+Ker}: local binning median estimator with kernel smoothing post-processing described in Section~\ref{sec:postprocessing}.
%	\item \textbf{LBM+LPR}: local binning median estimator with local polynomial regression described in Section~\ref{sec:mult_dim}.
%\end{itemize}
%\begin{itemize}
(1) \textbf{Kernel}: the classical kernel smoothing estimator;
(2) \textbf{T-kernel}: truncated kernel smoothing estimator described in Section~\ref{sec:naive}.
	We use an additional hyperparameter $T$ to control the truncation level;
	(3) \textbf{LBM}: local binning median estimator described in Section~\ref{sec:lbm};
	(4) \textbf{LBM+Ker}: local binning median estimator with kernel smoothing post-processing described in Section~\ref{sec:postprocessing};
	(5) \textbf{LPR}: standard local polynomial regression;
	(6) \textbf{LBM+LPR}: local binning median estimator with local polynomial regression described in Section~\ref{sec:mult_dim}.
%\end{itemize}
For all experiments, the hyperparameters are tuned to achieve the best performance.
For all figures, we only show 10\% observation points for better visualization.

In Figure~\ref{fig:d1beta05}, we consider estimating a one dimensional function with low smoothness.
We choose $\beta = 0.5$, $\epsilon=0.1$, $\rho = 0.2$ and $f(x) = L(x-\rho)^\beta$  for $x \le \rho$ and $f(x)=0$ otherwise. 
We let $Q$ be a Bernoulli distribution with half probability being $100$ and half probability being $-100$.
Figure~\ref{fig:d1beta05} shows our estimator is consistently better than other estimator.
Further when $L$ becomes bigger, truncated kernel estimator has worse performance, verifying our theoretical analysis in Section~\ref{sec:naive}.
On the other hand, local binning median estimator is not being affected.

In Figure~\ref{fig:d1beta15} we compare different estimators for estimating a one dimension function.
We use the same setup as in Figure~\ref{fig:d1beta05} except change the smoothness $\beta$ to $1.5$.
In this setting, naive algorithms like LPR and T-Kernel do not perform well while our proposed LBM+LPR and LBM+Ker give significant better results.

In Figure~\ref{fig:d2}, we consider estimating a peak function\footnote{https://www.mathworks.com/help/matlab/ref/peaks.html} using direct LPR and LBM+LPR.
Note the fitting by LPR is far from the true function whereas the estimation by our proposed LBM + LPR method is close to the truth.

Lastly, in Figure~\ref{fig:image_denoising} we explore our pre-processing procedure combining with other non-parametric estimator.
Here we consider the image denoising task where every pixel is subject to stochastic noise and a small amount of pixels are subject to adversarial noise.
Figure~\ref{fig:noisy_image} shows the noisy image.
In Figure~\ref{fig:directTV} we directly apply Total Variation de-noising algorithm~\citep{rudin1992nonlinear}.
However, due to the adversarial noise, there are still noisy points in the output image.
In Figure~\ref{fig:LBM_TV}, we first use local binning median then apply Total Variation de-noising algorithm.
Here, we successfully remove all adversarial noise.

\begin{figure*}[t!]
	\centering
	\begin{subfigure}[t]{0.25\textwidth}
		\includegraphics[width=\textwidth]{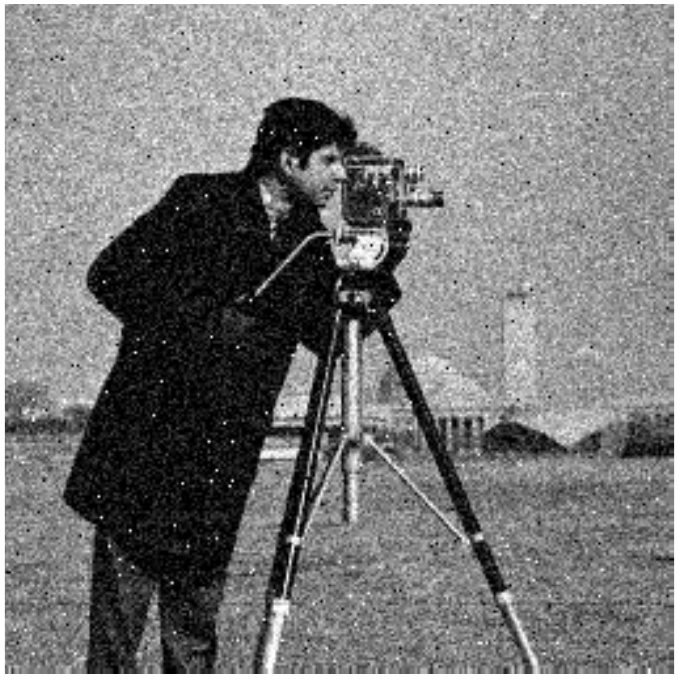}
		\caption{Noisy image. White points on the coat and black points in the background are adversarial points.}
		\label{fig:noisy_image}
	\end{subfigure}	
	\quad
	\begin{subfigure}[t]{0.25\textwidth}
		\includegraphics[width=\textwidth]{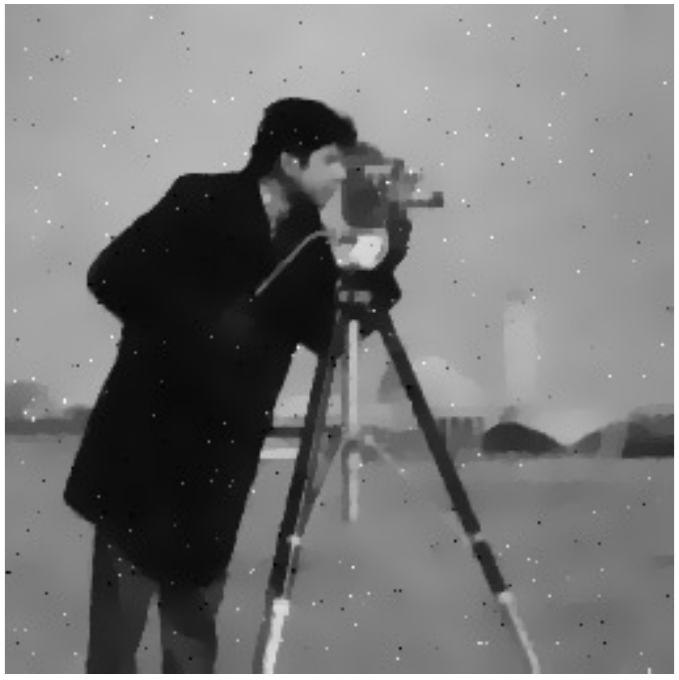}
		\caption{Direct total variation de-noising.}
		\label{fig:directTV}
	\end{subfigure}
	\quad
	\begin{subfigure}[t]{0.25\textwidth}
		\includegraphics[width=\textwidth]{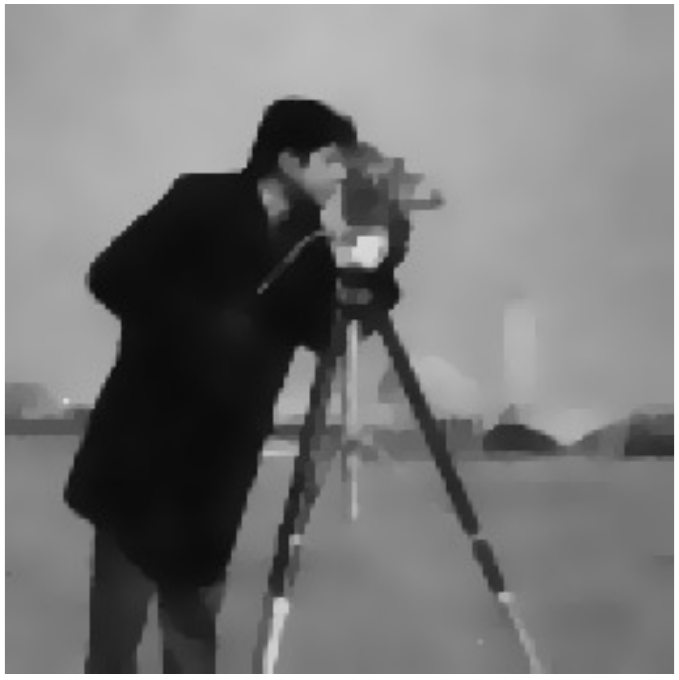}
		\caption{Local binning median followed total variation de-noising.}
		\label{fig:LBM_TV}
	\end{subfigure}
	\caption{Experimetns on image denoising.
	}
	\label{fig:image_denoising}
%	\vspace{-0.3cm}
\end{figure*}

%\section{Conclusion and Future Works}
%\label{sec:con}
%\input{conclusion.tex}

\bibliography{simonduref}

\begin{thebibliography}{42}
\providecommand{\natexlab}[1]{#1}
\providecommand{\url}[1]{\texttt{#1}}
\expandafter\ifx\csname urlstyle\endcsname\relax
  \providecommand{\doi}[1]{doi: #1}\else
  \providecommand{\doi}{doi: \begingroup \urlstyle{rm}\Url}\fi

\bibitem[Acharya et~al.(2017)Acharya, Diakonikolas, Li, and
  Schmidt]{acharya2017sample}
Jayadev Acharya, Ilias Diakonikolas, Jerry Li, and Ludwig Schmidt.
\newblock Sample-optimal density estimation in nearly-linear time.
\newblock In \emph{Proceedings of the Twenty-Eighth Annual ACM-SIAM Symposium
  on Discrete Algorithms}, pages 1278--1289. SIAM, 2017.

\bibitem[Balakrishnan et~al.(2017)Balakrishnan, Du, Li, and
  Singh]{balakrishnan2017computationally}
Sivaraman Balakrishnan, Simon~S Du, Jerry Li, and Aarti Singh.
\newblock Computationally efficient robust sparse estimation in high
  dimensions.
\newblock In \emph{Conference on Learning Theory}, pages 169--212, 2017.

\bibitem[Brown et~al.(2008)Brown, Cai, and Zhou]{brown2008robust}
Lawrence~D Brown, T~Tony Cai, and Harrison~H Zhou.
\newblock Robust nonparametric estimation via wavelet median regression.
\newblock \emph{The Annals of Statistics}, pages 2055--2084, 2008.

\bibitem[Cai et~al.(2009)Cai, Zhou, et~al.]{cai2009asymptotic}
T~Tony Cai, Harrison~H Zhou, et~al.
\newblock Asymptotic equivalence and adaptive estimation for robust
  nonparametric regression.
\newblock \emph{The Annals of Statistics}, 37\penalty0 (6A):\penalty0
  3204--3235, 2009.

\bibitem[Chan et~al.(2014)Chan, Diakonikolas, Servedio, and Sun]{chan2014near}
Siu~On Chan, Ilias Diakonikolas, Rocco~A Servedio, and Xiaorui Sun.
\newblock Near-optimal density estimation in near-linear time using
  variable-width histograms.
\newblock In \emph{Advances in Neural Information Processing Systems}, pages
  1844--1852, 2014.

\bibitem[Charikar et~al.(2017)Charikar, Steinhardt, and
  Valiant]{charikar2017learning}
Moses Charikar, Jacob Steinhardt, and Gregory Valiant.
\newblock Learning from untrusted data.
\newblock In \emph{Proceedings of the 49th Annual ACM SIGACT Symposium on
  Theory of Computing}, pages 47--60. ACM, 2017.

\bibitem[Chen et~al.(2015)Chen, Gao, and Ren]{chen2015robust}
Mengjie Chen, Chao Gao, and Zhao Ren.
\newblock Robust covariance matrix estimation via matrix depth.
\newblock \emph{arXiv preprint arXiv:1506.00691}, 2015.

\bibitem[Chen et~al.(2016)Chen, Gao, Ren, et~al.]{chen2016general}
Mengjie Chen, Chao Gao, Zhao Ren, et~al.
\newblock A general decision theory for huber’s $epsilon$-contamination
  model.
\newblock \emph{Electronic Journal of Statistics}, 10\penalty0 (2):\penalty0
  3752--3774, 2016.

\bibitem[Daskalakis et~al.(2012)Daskalakis, Diakonikolas, and
  Servedio]{daskalakis2012learning}
Constantinos Daskalakis, Ilias Diakonikolas, and Rocco~A Servedio.
\newblock Learning k-modal distributions via testing.
\newblock In \emph{Proceedings of the twenty-third annual ACM-SIAM symposium on
  Discrete Algorithms}, pages 1371--1385. Society for Industrial and Applied
  Mathematics, 2012.

\bibitem[Diakonikolas et~al.(2016{\natexlab{a}})Diakonikolas, Kamath, Kane, Li,
  Moitra, and Stewart]{diakonikolas2016robust}
Ilias Diakonikolas, Gautam Kamath, Daniel Kane, Jerry Li, Ankur Moitra, and
  Alistair Stewart.
\newblock Robust estimators in high dimensions without the computational
  intractability.
\newblock \emph{arXiv preprint arXiv:1604.06443}, 2016{\natexlab{a}}.

\bibitem[Diakonikolas et~al.(2016{\natexlab{b}})Diakonikolas, Kane, and
  Stewart]{diakonikolas2016efficient}
Ilias Diakonikolas, Daniel~M Kane, and Alistair Stewart.
\newblock Efficient robust proper learning of log-concave distributions.
\newblock \emph{arXiv preprint arXiv:1606.03077}, 2016{\natexlab{b}}.

\bibitem[Donoho and Johnstone(1994)]{donoho1994ideal}
David~L Donoho and Jain~M Johnstone.
\newblock Ideal spatial adaptation by wavelet shrinkage.
\newblock \emph{Biometrika}, 81\penalty0 (3):\penalty0 425--455, 1994.

\bibitem[Donoho et~al.(1998)Donoho, Johnstone, et~al.]{donoho1998minimax}
David~L Donoho, Iain~M Johnstone, et~al.
\newblock Minimax estimation via wavelet shrinkage.
\newblock \emph{The Annals of Statistics}, 26\penalty0 (3):\penalty0 879--921,
  1998.

\bibitem[Fan(1993)]{fan1993local}
Jianqing Fan.
\newblock Local linear regression smoothers and their minimax efficiencies.
\newblock \emph{The Annals of Statistics}, 21\penalty0 (1):\penalty0 196--216,
  1993.

\bibitem[Fan and Gijbels(1992)]{fan1992variable}
Jianqing Fan and Ir\`{e}ne Gijbels.
\newblock Variable bandwidth and local linear regression smoothers.
\newblock \emph{The Annals of Statistics}, 20\penalty0 (4):\penalty0
  2008--2036, 1992.

\bibitem[Fan and Gijbels(1996)]{fan1996local}
Jianqing Fan and Irene Gijbels.
\newblock \emph{Local polynomial modelling and its applications}.
\newblock CRC Press, 1996.

\bibitem[Fan et~al.(1994)Fan, Hu, and Truong]{fan1994robust}
Jianqing Fan, Tien-Chung Hu, and Young~K Truong.
\newblock Robust non-parametric function estimation.
\newblock \emph{Scandinavian journal of statistics}, pages 433--446, 1994.

\bibitem[Friedman et~al.(2001)Friedman, Hastie, and
  Tibshirani]{friedman2001elements}
Jerome Friedman, Trevor Hastie, and Robert Tibshirani.
\newblock \emph{The elements of statistical learning}, volume~1.
\newblock Springer series in statistics New York, 2001.

\bibitem[Gao(2017)]{gao2017robust}
Chao Gao.
\newblock Robust regression via mutivariate regression depth.
\newblock \emph{arXiv preprint arXiv:1702.04656}, 2017.

\bibitem[Geer(2000)]{geer2000empirical}
Sara~A Geer.
\newblock \emph{Empirical Processes in M-estimation}, volume~6.
\newblock Cambridge university press, 2000.

\bibitem[Green and Silverman(1993)]{green1993nonparametric}
Peter~J Green and Bernard~W Silverman.
\newblock \emph{Nonparametric regression and generalized linear models: a
  roughness penalty approach}.
\newblock CRC Press, 1993.

\bibitem[Gy{\"o}rfi et~al.(2006)Gy{\"o}rfi, Kohler, Krzyzak, and
  Walk]{gyorfi2006distribution}
L{\'a}szl{\'o} Gy{\"o}rfi, Michael Kohler, Adam Krzyzak, and Harro Walk.
\newblock \emph{A distribution-free theory of nonparametric regression}.
\newblock Springer Science \& Business Media, 2006.

\bibitem[Hampel et~al.(2011)Hampel, Ronchetti, Rousseeuw, and
  Stahel]{hampel2011robust}
Frank~R Hampel, Elvezio~M Ronchetti, Peter~J Rousseeuw, and Werner~A Stahel.
\newblock \emph{Robust statistics: the approach based on influence functions},
  volume 114.
\newblock John Wiley \& Sons, 2011.

\bibitem[H{\"a}rdle et~al.(2012)H{\"a}rdle, Kerkyacharian, Picard, and
  Tsybakov]{hardle2012wavelets}
Wolfgang H{\"a}rdle, Gerard Kerkyacharian, Dominique Picard, and Alexander
  Tsybakov.
\newblock \emph{Wavelets, approximation, and statistical applications}, volume
  129.
\newblock Springer Science \& Business Media, 2012.

\bibitem[Hastings~Jr et~al.(1947)Hastings~Jr, Mosteller, Tukey, and
  Winsor]{hastings1947low}
Cecil Hastings~Jr, Frederick Mosteller, John~W Tukey, and Charles~P Winsor.
\newblock Low moments for small samples: a comparative study of order
  statistics.
\newblock \emph{The Annals of Mathematical Statistics}, pages 413--426, 1947.

\bibitem[Huber(2011)]{huber2011robust}
Peter~J Huber.
\newblock \emph{Robust statistics}.
\newblock Springer, 2011.

\bibitem[Huber et~al.(1964)]{huber1964robust}
Peter~J Huber et~al.
\newblock Robust estimation of a location parameter.
\newblock \emph{The Annals of Mathematical Statistics}, 35\penalty0
  (1):\penalty0 73--101, 1964.

\bibitem[Huber et~al.(1965)]{huber1965robust}
Peter~J Huber et~al.
\newblock A robust version of the probability ratio test.
\newblock \emph{The Annals of Mathematical Statistics}, 36\penalty0
  (6):\penalty0 1753--1758, 1965.

\bibitem[Kleinberg and Tardos(2006)]{kleinberg2006algorithm}
Jon Kleinberg and Eva Tardos.
\newblock \emph{Algorithm design}.
\newblock Pearson Education India, 2006.

\bibitem[Lai et~al.(2016)Lai, Rao, and Vempala]{lai2016agnostic}
Kevin~A Lai, Anup~B Rao, and Santosh Vempala.
\newblock Agnostic estimation of mean and covariance.
\newblock \emph{arXiv preprint arXiv:1604.06968}, 2016.

\bibitem[Larry(2006)]{larry2006all}
Wasserman Larry.
\newblock \emph{All of nonparametric statistics}.
\newblock Springer Texts in Statistics. New York: Springer Science+ Business
  Media, 2006.

\bibitem[Liu and Gao(2017)]{liu2017density}
Haoyang Liu and Chao Gao.
\newblock Density estimation with contaminated data: Minimax rates and theory
  of adaptation.
\newblock \emph{arXiv preprint arXiv:1712.07801}, 2017.

\bibitem[Nemirovski(2000)]{nemirovski2000topics}
Arkadi Nemirovski.
\newblock Topics in non-parametric.
\newblock 2000.

\bibitem[Reinsch(1967)]{reinsch1967smoothing}
Christian~H Reinsch.
\newblock Smoothing by spline functions.
\newblock \emph{Numerische Mathematik}, 10\penalty0 (3):\penalty0 177--183,
  1967.

\bibitem[Rudin et~al.(1992)Rudin, Osher, and Fatemi]{rudin1992nonlinear}
Leonid~I Rudin, Stanley Osher, and Emad Fatemi.
\newblock Nonlinear total variation based noise removal algorithms.
\newblock \emph{Physica D: nonlinear phenomena}, 60\penalty0 (1-4):\penalty0
  259--268, 1992.

\bibitem[Ruppert(2011)]{ruppert2011statistics}
David Ruppert.
\newblock \emph{Statistics and data analysis for financial engineering},
  volume~13.
\newblock Springer, 2011.

\bibitem[Tsybakov(2009)]{tsybakov2009introduction}
Alexandre~B Tsybakov.
\newblock \emph{Introduction to nonparametric estimation.}
\newblock Springer Series in Statistics. Springer, New York, 2009.

\bibitem[Tukey(1975)]{tukey1975mathematics}
John~W Tukey.
\newblock Mathematics and the picturing of data.
\newblock In \emph{Proceedings of the international congress of
  mathematicians}, volume~2, pages 523--531, 1975.

\bibitem[Wang et~al.(2008)Wang, Brown, Cai, and Levine]{wang2008effect}
Lie Wang, Lawrence~D Brown, T~Tony Cai, and Michael Levine.
\newblock Effect of mean on variance function estimation in nonparametric
  regression.
\newblock \emph{The Annals of Statistics}, pages 646--664, 2008.

\bibitem[Wang et~al.(2018)Wang, Balakrishnan, and Singh]{wang2018optimization}
Yining Wang, Sivaraman Balakrishnan, and Aarti Singh.
\newblock Optimization of smooth functions with noisy observations: Local
  minimax rates.
\newblock \emph{arXiv preprint arXiv:1803.08586}, 2018.

\bibitem[Whittaker(1922)]{whittaker1922new}
Edmund~T Whittaker.
\newblock On a new method of graduation.
\newblock \emph{Proceedings of the Edinburgh Mathematical Society},
  41:\penalty0 63--75, 1922.

\bibitem[Yatracos et~al.(1985)]{yatracos1985rates}
Yannis~G Yatracos et~al.
\newblock Rates of convergence of minimum distance estimators and kolmogorov's
  entropy.
\newblock \emph{The Annals of Statistics}, 13\penalty0 (2):\penalty0 768--774,
  1985.

\end{thebibliography}
\bibliographystyle{plainnat}
\onecolumn
\newpage
%\newpage
\appendix
\section{Proofs}
\subsection{Useful Lemmas}
\label{sec:useful_lemmas}
We first establish the following lemma that provides the key bias-variance decomposition of the local binning median estimator.
The main component is a deterministic analysis on the median and the noise structure.
\begin{lem}\label{lem:fixed_design_median}
	Denote $z_\vect{j} = \median\left\{y_\vect{i}\right\}_{\vect{i} \in \text{ bin }\vect{j}}$ as the median estimator in the bin $\vect{j}$. 
	Then $z_\vect{j}$ can be written as $z_\vect{j} = f(\vect{j}/m) + \eta_\vect{j} + \Delta_\vect{j}$, 
	%\begin{align}
	%z_j = f\left(\frac{j}{m}\right) + \eta_j + \triangle_j
	%\end{align}
	where \begin{itemize}
		\item $\eta_\vect{j} = \median\left\{\xi_\vect{i}\right\}_{\vect{i} \in \text{ bin }\vect{j}}$ where $\xi_{\vect{i}} \sim (1-\epsilon)N(0,1)+\epsilon (Q(x_{\vect{i}})-f(\frac{\vect{i}}{n}))$\footnote{Here$Q(x_{\vect{i}})-f(\frac{\vect{i}}{n})$ is a distribution that shifts $Q$ by $-f(\frac{\vect{i}}{n})$.}, and
		\item $\abs{\triangle_\vect{j}} \leq\max_{\vect{i} \in \text{ bin }\vect{j}} \abs{f(\frac{\vect{i}}{n})-f(\frac{\vect{j}}{n})}$ almost surely.
	\end{itemize} 
\end{lem}
To prove Lemma \ref{lem:fixed_design_median}, we need the following sandwiching inequality for the median operator.
\begin{prop}
For any sequences $\{a_i\}_i$ and $\{b_i\}_i$ of equal length, it holds that
$
\min\{a_i\}_i \leq \median\{a_i+b_i\}_i - \median\{b_i\}_i \leq \max\{a_i\}_i.
$
\label{prop:median_difference}
\end{prop}
\begin{proof}
	\begin{align*}
	&\median\left(\left\{Y_i\right\}_{i=1}^m\right)-\median \left(\left\{\xi_i\right\}_{i=1}^m\right) \\
	=& \median\left(\left\{Z_i+\xi_i\right\}_{i=1}^m\right)-\median \left(\left\{\xi_i\right\}_{i=1}^m\right) \\
	\ge & \median\left(\min Z_i+\left\{\xi_i\right\}_{i=1}^m\right)-\median \left(\left\{\xi_i\right\}_{i=1}^m\right)  \\
	= &\min Z_i.
	\end{align*}
	Using similar argument we can prove the other direction.
\end{proof}

\begin{proof}[Proof of Lemma~\ref{lem:fixed_design_median}]
	Recall we can write the observation model as\[
	y_\vect{i} = f(\vect{i}/n)+\xi_\vect{i} \]where $\xi_i \sim \left(1-\epsilon\right) N\left(0,1\right)+\epsilon Q\left(x_{\vect{i}}-f(\frac{\vect{i}}{n})\right)$.
	Therefore,
	\begin{align*}
	y_\vect{i} - f(\vect{j}/m) = \left(f(\vect{i}/n)-f(\vect{j}/m)\right)  + \xi_\vect{i}.
	\end{align*}
	Denote $\eta_\vect{j} = \median \left\{\xi_i\right\}_{\vect{i} \in \text{ bin }\vect{j}}$.
	Applying Proposition~\ref{prop:median_difference}  and noting that $z_\vect{j}=\median\{y_i\}_{\vect{i} \in \text{ bin }\vect{j}}$, we have \begin{align*}
	\min_{\vect{i} \in \text{ bin }\vect{j}}f\left(\vect{i}/n\right)\le z_\vect{j} - \eta_\vect{j} \le \max_{\vect{i} \in \text{ bin }\vect{j}}f\left(\vect{i}/n\right).
	\end{align*}
	Define $\triangle_\vect{j} = z_\vect{j} - \eta_\vect{j} - f\left(\frac{\vect{j}}{m}\right)$.
We have
	\begin{align*}
	\abs{\triangle_j} \le \max_{\vect{i} \in \text{ bin }\vect{j}} \abs{f(\vect{i}/n)-f(\vect{j}/m)}
	\end{align*}
	almost surely. The lemma is thus proved.
\end{proof}

To analyze $\eta_\vect{j}$, we use a decoupled analysis for the adversarial noise and the stochastic noise.
Suppose out of the $s$ samples in the $\vect{j}$-th bin, $s_\vect{j}$ observations come from the adversarial noise distribution $Q$.
Our key lemma shows that if  $s_\vect{j} \le s/4$, these adversarial observations incur a small amount of additional bias in the median over $\{\xi_\vect{i}\}_{\vect{i} \in \text{ bin }\vect{j}}$.

\begin{lem}\label{lem:worst_case_adv}
	Suppose $0 \le s' < s/4$. Let $\tilde\xi_1,\ldots,\tilde\xi_{s-s'}$ be fixed and $\tilde\xi_{s-s'+1},\ldots,\tilde\xi_{s}$ be arbitrary,
	corresponding to the $s$ noise variables $\{\xi_i\}_{i=(j-1)s+1}^{js}$.
	Then 
	\begin{align*}
	\sup_{\tilde\xi_{s-s'+1},\ldots,\tilde\xi_{s}}\median\{\tilde\xi_i \}_{i=1}^s &\le \tilde\xi_{(\frac{s}{2}:s-s')};\\
	\inf_{\tilde\xi_{s-s'+1},\ldots,\tilde\xi_{s}}\median\{\tilde\xi_i \}_{i=1}^s &\ge \tilde\xi_{(\frac{s-2s'}{2}:s-s')},
	\end{align*} where $\tilde\xi_{(\frac{s}{2}:s-s')}$ and $\tilde\xi_{(\frac{s-2s'}{2}:s-s')}$ are the $\frac{s}{2}$-th and the $\frac{s-2s'}{2}$-th largest elements in $\tilde\xi_1,\ldots,\tilde\xi_{s-s'}$, respectively. 
\end{lem}
\begin{proof}[Proof of Lemma~\ref{lem:worst_case_adv}]
	The median is maximized by setting $\xi_{s-s'+1},\ldots,\xi_{s} \ge \max\left\{\xi_i\right\}_{i=1}^{s-s'}$ and this gives us the first inequality.
	The second inequality can be proved in a similar manner.
\end{proof}

As a corollary, conditioned on the event that $s_j<s/4$, 
the bias and variance in $\eta_j$ can be upper bounded,
following standard properties of the order statistics~\citep{ruppert2011statistics}.
\begin{cor}
Suppose $s_\vect{j}<s/4$ for all $\vect{j}\in[m]^d$. Then there exists an absolute constant $C>0$ such that for all $\vect{j}$, 
\begin{equation*}
\big|\expect{\eta_j}\big| \leq Cs_\vect{j}/s \;\;\;\text{and}\;\;\; \variance[\eta_\vect{j}] \leq C/s.
\end{equation*}
\label{cor:quantile-gaussian}
\end{cor}

Both Lemma \ref{lem:worst_case_adv} and Corollary \ref{cor:quantile-gaussian} depend crucially on the condition $s_\vect{j}<s/4$, that at most one quarter of the observations within each local bin are corrupted by adversarial noise.
This is likely to be satisfies when $\epsilon$ is not too large (e.g., $\epsilon\ll 1/4$) because adversarial noise samples are uniformly distributed across all samples.
The following lemma gives a rigorous statement of the above intuition: %, which is proved in the appendix.
\begin{lem}[Uniform Upper Bound for $s_j$s]\label{lem:uniform_upper_sj}
With high probability, for all $\vect{j} \in [m]^d$ we have \begin{align*}
	s_{\vect{j}} \le C\left(s\epsilon + \log m\right)
\end{align*}
for some $C>0$.
\end{lem}
\begin{proof}[Proof of Lemma~\ref{lem:uniform_upper_sj}]
For each $\vect{j}$, by Chernoff bound we have \begin{align*}
	\prob\left(s_\vect{j} \ge \left(1+\delta\right)\epsilon s\right) \le \exp\left(\frac{-\delta^2s\epsilon}{2+\delta}\right).
\end{align*}
Choose $\delta = C\left(1+\frac{\log m}{s\epsilon}\right)$ for some large enough $C$.
If $s\epsilon \ge \log m$, we have \begin{align*}
&\prob\left(s_\vect{j} \ge \left(1+\delta\right)\epsilon\right) \\
\le &\exp\left(\frac{-\delta\epsilon s}{2}\right) \\
\le & \exp\left(\frac{-\delta\log m}{2}\right) \\
\le & \frac{1}{C_1 m}
\end{align*}
for some $C_1 \ge 100$.
If $s\epsilon \le \log m$, we have \begin{align*}
&\prob\left(s_\vect{j} \ge \left(1+\delta\right)\epsilon\right) \\
\le &\exp\left(\frac{-\delta\epsilon s}{2}\right) \\
\le & \exp\left(\frac{-\log m}{2}\right) \\
\le & \frac{1}{C_1 m}.
\end{align*}
Now using union bound we obtain the desired result.
\end{proof}

%\begin{lem}\label{lem:max_adv_in_each_bin}
%	%Suppose $
%	%Let $s_j$ be number of adversary points in the $j$-th bin,
%	 Suppose $\epsilon \le 1/8$. 
%	 If $s \ge C\log m$ for some absolute constant $C>0$ then with probability at least 0.99 uniformly over all $j\in[m]$, $s_j \le s/4$.
%%	 {\color{red}\bf [Finish here]}
%\end{lem}
%\begin{proof}[Proof of Lemma~\ref{lem:max_adv_in_each_bin}]
%	%Note number of adversarial noise follows a binomial distribution: $s_j \sim \text{ bin}\left(s,\epsilon\right)$.
%	Note that $s_j$ follows a binomial distribution $s_j\sim \text{Bin}(s,\epsilon)$.
%	Applying Hoeffding's inequality,
%	\begin{align*}
%	\Pr\left[s_j \ge s\epsilon + \frac{s}{4}-\epsilon s\right] &\le \exp\left\{-\frac{2s^2\left(1/4-\epsilon\right)}{s}\right\} \\
%	&\le \frac{1}{100m}.
%	\end{align*}
%	The lemma is proved by a union bound over all $j\in[m]$.
%	%Using union bound over all bins, we obtain the desired result.
%\end{proof}

\subsection{Proof of Lemma~\ref{lem:meta_lemma_uniform_mult}}
\label{sec:meta_lemma_proof}
Applying Lemma~\ref{lem:uniform_upper_sj}, using the decomposition Lemma~\ref{lem:fixed_design_median} and then using Corollary~\ref{cor:quantile-gaussian} we directly obtain Lemma~\ref{lem:meta_lemma_uniform_mult}.

\subsection{Proof of Theorem~\ref{thm:l2_upper_bound}}
\label{sec:proof_median_holder}
Consider query point $x$ and let $\vect{j}$ be the local bin $x$ belong to. 
Recall that the local binning median estimate $\hat f(x)$ is equal to $z_{\vect{j}}=\median\{y_\vect{i}\}_{\vect{i} \in \text{ bin }\vect{j}}$.
We then have 
\begin{align}
%&\expect{\left(\hat{f}\left(x\right)-f\left(x\right)\right)^2} \nonumber \\
&\expect{
	(z_\vect{j}-f(x))^2
	}=\expect{
	\left(f(\vect{j}/m)+\eta_\vect{j}+\triangle_\vect{j} - f(x)\right)^2
	}\nonumber\\
&\le  3\expect{\left(f\left(\vect{j}/m\right)-f(x)\right)^2 + \eta_\vect{j}^2+\triangle_\vect{j}^2}.\nonumber
%\le &3\left(L^2\left(\frac{1}{m}\right)^{2\beta} + C\left(\frac{s_j^2}{s^2} + \frac{1}{s}\right) + L^2\left(\frac{1}{m}\right)^{2\beta} 
%\right)\nonumber \\
%\triangleq & C_1\left(L^2\left(\frac{1}{m}\right)^{2\beta} +\frac{s_j^2}{s^2} + \frac{1}{s}\right). \label{eqn:pointwise_risk}
\end{align}
%The equalities are due to Theorem~\ref{thm:fixed_design_median}.
%The first inequality we use $(a+b+c)^2 \le 3(a^2+b^2+c^2)$.
%The second inequality we use Theorem~\ref{thm:eta_bound} and the assumption that $f$ belongs to the \holder function class.
Here the last inequality holds because $(a+b+c)^2\leq 3(a^2+b^2+c^2)$.
Invoking Lemma \ref{lem:worst_case_adv} that upper bounds $\Delta_\vect{j}^2$ and Corollary \ref{cor:quantile-gaussian} that upper bounds $\expect{\eta_\vect{j}^2}$, 
we know that 
\begin{align*}
\expect{(z_\vect{j}-f(x))^2}
\leq C_1\left[\triangle_\vect{j}^2+ \frac{s \epsilon^2 +\log^2 m}{s^2} + \frac{1}{s} \right],
\end{align*}
where $C_1>0$ is an absolute constant. Subsequently, 
%Now we can bound the expected $\ell_2$ risk.
\begin{align}
\expect{\|\hat{f}-f\|_2^2} 
%=  \expect{\int_{0}^{1}\left(\hat{f}(x)-f(x)\right)^2dx} \nonumber \\
 &\leq \sum_{\vect{j}\in [m]^d}\int_{\text{ bin }\vect{j}}\expect{(\hat{f}(x)-f(x))^2}dx \nonumber\\
  &\leq C_2\left[ \frac{\sum_{\vect{j}\in[m]^d}\triangle_\vect{j}^2}{m^d} + \epsilon^2+\frac{1}{s}\right] \nonumber \\
 &\leq C_1\left[ \frac{L^2}{m^{2\beta}} + \epsilon^2+\frac{1}{s}\right].\label{eq:l2-risk}
 %\le & \frac{1}{m} \sum_{j=1}^{m}C_1\left(L^2\left(\frac{1}{m}\right)^{2\beta} +\frac{s_j^2}{s^2} + \frac{1}{s}\right) \nonumber\\
% = & C_1\left(
%L^2\left(\frac{1}{m}\right)^{2\beta} + \frac{\sum_{j=1}^{m}s_j^2}{ms^2}+\frac{1}{s}
%\right) \label{eqn:l2_risk}
\end{align} 
for some $C_2 > 0$.
%where for the inequality we have used Equation~\eqref{eqn:pointwise_risk}.
Setting $m\asymp n^{\frac{1}{d+2\beta}}L^{\frac{2}{2\beta+1}}$
we proved Theorem \ref{thm:l2_upper_bound}.

%Notice that plugging $m\asymp n^{\frac{1}{1+2\beta}}L^{\frac{2}{2\beta+1}}$, we have \begin{align*}
%L^2\left(\frac{1}{m}\right)^{2\beta} + \frac{1}{s} \asymp  L^{\frac{2}{2\beta+1}}n^{-\frac{2\beta}{2\beta+1}},
%\end{align*} which is the desired statistical rate.
%
%\begin{lem}\label{lem:square_adv_in_each_bin}
%	Let $s_j$ be the number of adversary in the $j$-th bin, then with high probability, \begin{align*}
%	\sum_{j=1}^{n}s_j^2 \le C\left(\epsilon s m + s^2\epsilon^2 m\right)
%	\end{align*} for some absolute constant $C$.
%	\begin{proof}[Proof of Lemma~\ref{lem:square_adv_in_each_bin}]
%		Note that $s_j \sim \text{ bin}\left(s,\epsilon\right)$ so we have
%		\begin{align*}
%		\expect{\sum_{j=1}^{m}s_j^2} = m\left(s^2\epsilon^2+s\epsilon\right).
%		\end{align*}
%		Now apply Markov's inequality we obtain the desired result.
%	\end{proof}
%\end{lem}

\subsection{Proof of Theorem \ref{thm:median_kernel}}

To analyze the kernel smoothing post-processing step, 
we need the following technical lemma, which shows that $K_j^h(\cdot)$ sums to one and is therefore a valid kernel.
\begin{lem}\label{lem:sum_kernel}
For any $x \in \left(h, 1-h\right)$, $\sum_{j=1}^{m}K_j^h(x)=1$.
\end{lem}
\begin{proof}[Proof of Lemma~\ref{lem:sum_kernel}]
Recall the definition \begin{align*}
		\sum_{j=1}^mK_{j}^h\left(x\right) = \frac{1}{h}\sum_{j=1}^{m}\int_{\frac{j-1}{m}}^{\frac{j}{m}}K\left(\frac{x-u}{h}\right)du.
	\end{align*}
Setting $v = \frac{x-u}{h}$, we have \begin{align*}
		\sum_{j=1}^mK_{j}^h\left(x\right) = & \frac{1}{h}\sum_{j=1}^{m}\int_{\frac{x-(j-1)/m}{h}}^{\frac{x-j/m}{h}}K(v)(-h)dv \\
		= &\sum_{j=1}^{m} \int_{\frac{(j-1)/m-x}{h}}^{\frac{j/m-x}{h}}K(v) dv \\
		= & \int_{\frac{-x}{h}}^{\frac{1-x}{h}}K(v)dv \\
		= & 1.
\end{align*}
\end{proof}

We are now ready to prove Theorem \ref{thm:median_kernel}.
Let $x_0\in[c,1-c]$ be an interior query point at which estimation of $f(x_0)$ is sought.
By standard bias-variance decomposition, the point-wise mean-square error $\mathbf E[(\check f(x_0)-f(x_0))^2]$
equals
%We first use bias-variance decomposition to expand the error term \begin{align*}
$$
\abs{\expect{\check{f}\left(x_0\right)}-f(x_0)}^2 + \variance\left(\check{f}(x_0)\right).
$$
%Recall by Theorem~\ref{thm:fixed_design_median}, with high probability, we can write \[
%z_j = f(\frac{j}{m}) + \eta_j + \triangle_j 
%\] where \begin{align*}
%\abs{\triangle_j} \le &\max_{(j-1)s+1\le i \le js} \abs{f\left(\frac{i}{n}\right)-f\left(\frac{j}{m}\right)},  \\
%\abs{\expect{\eta}_j} \le & C\frac{s_j}{s}, \\
%\variance\left(\eta_j\right) \le & C\frac{1}{s}.
%\end{align*}

Recall the decomposition that $z_j=f(j/m)+\eta_j+\Delta_j$,
where $\{z_j\}_{j=1}^m$ are local binning medians used as inputs of the kernel smoothing estimator $\check f(x_0)=\sum_{j=1}^mK_j^h(x_0)z_j$.
%We first analyze the bias term
By triangle inequality, the bias term $|\expect{\check f(x_0)}-f(x_0)|$ can be upper bounded by
\begin{align*}
%\abs{\expect{\hat{f}\left(x_0\right)}-f(x_0)} \le
 \sum_{j=1}^{m}K_j^h(x_0)\abs{f(x_j)-f(x_0)} + \sum_{j=1}^{m}K_j^h(x_0)\left[\abs{\triangle_j}+\expect{\eta_j}\right].
\end{align*}

The first term is the standard bias term in (non-robust) kernel smoothing. 
Using arguments in~\cite{wang2008effect}, we have 
\begin{align*}
	\sum_{j=1}^{m}K_j^h(x_0)\abs{f(x_j)-f(x_0)} \le \frac{L}{m} + Lh^{\beta}.
\end{align*}
For the other term, by Lemma~\ref{lem:sum_kernel} and Lemma~\ref{lem:meta_lemma_uniform_mult}, we know that there exists an absolute constant $C>0$ such that
\begin{align*}
\sum_{j=1}^{m}K_j^h(x)\left(\abs{\triangle_j}+\expect{\eta_j}\right) \le &C\left[ \max_{j}\max_{(j-1)s<i\leq js}\big| f(i/n)-f(j/m)\big| + \max_j \frac{s_j}{s} \right]\\
\leq &C\left[\frac{L}{m} + \epsilon+\frac{\log m}{s}\right].
\end{align*}
%for some absolute constant $C > 0$.
%Because $f\in\Lambda(\beta,L)$ and $s/n\leq 1/m$, we have
Here the last inequality holds because 
$|f(i/n)-f(j/m)|\leq L/m$ for all $j\in[m]$ and $(j-1)s<i\leq js$ which is implied by $f\in\Lambda(\beta,L)$,
and $s_j\leq O(s\epsilon+\log m)$ thanks to Lemma \ref{lem:uniform_upper_sj}.
Subsequently,
%By our assumption of function class, we know \[
%\max_j \max_{(j-1)s+1\le i \le js}\abs{f\left(\frac{i}{n}\right)-f\left(\frac{j}{m}\right)} \le \frac{L}{m}.
%\]
%Combining with Lemma~\ref{lem:uniform_upper_sj}, we can bound the bias term as
\begin{align*}
\abs{\expect{\check{f}\left(x_0\right)}-f(x_0)} \le C\left[\epsilon +\frac{L}{m} + \frac{\log m}{s} + Lh^{\beta}\right].
\end{align*}

We next consider the variance term $\variance(\eta_j)$.
Since the kernel weights $K_j^h(x_0)$ are statistically independent of $\{\eta_j\}$, 
\begin{align*}
\variance\left(\check{f}(x_0)\right) = \sum_{i=1}^{m}\left(K_j^h(x_0)\right)^2\variance\left(\eta_j\right). %\le C\cdot\frac{1}{mh}\cdot\frac{1}{s} = \frac{C}{nh}.
\end{align*}
Using properties of the kernel $K_j^h(\cdot)$ and invoking Corollary \ref{cor:quantile-gaussian},
we have (again conditioned on $\{s_j\}_{j=1}^m$ and the event $s_j<s/4$ for all $j\in[m]$)
$$
\variance(\check f(x_0)) \leq C\cdot\frac{1}{mh}\cdot\frac{1}{s} = \frac{C}{nh}.
$$
Putting things together, the point-wise mean-square error $\mathbf E[(\check f(x_0)-f(x_0))^2]$ is upper bounded by (with probability 0.9)
\begin{align*}
%\expect{\hat{f}(x_0)-f(x_0)}^2 \le
 C\left[\epsilon^2 + \frac{L^2}{m^2} +\frac{\log^2 m}{s^2} +L^2h^{2\beta} + \frac{1}{nh}\right].
\end{align*}
Setting $m\asymp \sqrt{n}/{\sqrt[4]{\log n}} $, $h\asymp \left(nL^2\right)^{-\left(2\beta+1\right)}$
and integrating over $[c,1-c]$ we complete the proof of Theorem \ref{thm:median_kernel}.
%Now plugging in $m$ and $h$ and integrate over $\left[0,1\right]$ we obtain the desired rate.

\subsection{Proof of Theorem~\ref{thm:multivariate_lpr}}
To analyze this estimator, for $ x \in \left(h,1-h\right)^d$, define mapping $\psi_{x,h}(z) = \left(1,\psi_{x,h}^1(z),\ldots,\psi_{x,h}^\ell(z)\right) \in \mathbb{R}^D$ with $D = 1+d+\ldots+d^\ell$ where $\psi_{x,h}^j(z) = \left[\Pi_{\ell=1}^jh^{-1}\left(z_{i_\ell}-x_{i_\ell}\right)\right]^d_{i_1,\ldots,i_j=1}$ is the degree-$j$ polynomial mapping from $\mathbb{R}^d$ to $\mathbb{R}^{d_j}$.
Further define $\Psi_{t,h}  \in \mathbb{R}^{n_h \times D}$ aggregated design matrix where $n_h$ is the number of design points in $B_h^{\infty}(x)$.
Using these notations we can write the estimation in a compact form:
\begin{align*}
\hat{f}_h(z) = \psi_{x,h}(z)^\top \left(\Psi_{t,h}^\top \Psi_{t,h}\right)^{-1}Z_{t,h}
\end{align*} where $Z_{t,h} = \left(z_{\vect{j}}\right)_{\vect{j}/m \in B_h^{\infty}(x)}$.

The following Lemma characterizes the key property of the quantity $\psi_{x,h}(z)^\top \left(\Psi_{t,h}^\top \Psi_{t,h}\right)^{-1}$.
\begin{lem}\label{lem:key_lemma_lpr}
For any $x \in \left(0,1\right)^d$, we have 
\[\norm{\psi_{x,h}(x)^\top \left(\Psi_{x,h}^\top \Psi_{x,h}\right)^{-1}}_1 = O\left(1\right).\]
\end{lem}
\begin{proof}[Proof of Lemma~\ref{lem:key_lemma_lpr}]
First notice that $\norm{\psi_{x,h}(x)}_\infty = O(n_h)$ by definition.
Next, viewing the summation as the approximation to Riemann integral, we have that $\frac{1}{n_h}\left(\Psi_{x,h}^\top \Psi_{x,h}\right) = \int_{\left[0,1\right]^d} \psi_{0,1}(z)\psi_{0,1}(z)^\top d \lambda(z)+ O(\frac{n_h}{m})$ where $\lambda(\cdot)$ is the standard Lebesgue measure.
By proposition 7 of  \cite{wang2018optimization}, we have \[
\int_{\left[0,1\right]^d} \psi_{0,1}(z)\psi_{0,1}(z)^\top d \lambda(z) \succcurlyeq \Omega(1)\mat{I}.
\]
Thus we have $\frac{1}{n_h}\left(\Psi_{x,h}^\top \Psi_{x,h}\right)\succcurlyeq  \Omega(1)\mat{I}$.
Therefore, we can bound the spectral norm $\norm{\left(\Psi_{x,h}^\top \Psi_{x,h}\right)^{-1}}_2 = O\left(n_h\right)$.
Combing these two observations we have $\norm{\psi_{x,h}(x)^\top \left(\Psi_{x,h}^\top \Psi_{x,h}\right)^{-1}}_1 = O\left(1\right)$.
\end{proof}

For any query point $x$, we can write the its function value estimate as \begin{align*}
\hat{f}_h(x) = &\psi_{x,h}(x)^\top \left(\Psi_{x,h}^\top \Psi_{x,h}\right)^{-1}Z_{t,h} \\
= & \psi_{x,h}(x)^\top \left(\Psi_{x,h}^\top \Psi_{x,h}\right)^{-1}\begin{pmatrix}
\ldots \\
f\left(\frac{\vect{j}}{m}\right) + \triangle_\vect{j} + \eta_\vect{j}\\
\ldots
\end{pmatrix} \\
= & \psi_{x,h}(x)^\top \left(\Psi_{x,h}^\top \Psi_{x,h}\right)^{-1}\left[\begin{pmatrix}
\ldots \\
f\left(\frac{\vect{j}}{m}\right) + \triangle_\vect{j} + \left(\eta_\vect{j}-\expect{\eta_{\vect{j}}\right)}\\
\ldots
\end{pmatrix}
+
\begin{pmatrix}
\ldots \\
\triangle_\vect{j} + \expect{\eta_{\vect{j}}}\\
\ldots
\end{pmatrix}
\right]
\end{align*}
For \holder class, we know $\abs{\Delta_j} \le \frac{L}{m}$.
For Sobolev class, since we assume $\frac{\beta-1}{d} \ge \frac{1}{p}$, by Sobolev embedding theorem, we have $\abs{\Delta_j} \le \frac{L}{m}$ as well.
We also know $\abs{\expect{\eta_\vect{j}}} \le C\left(\epsilon + \frac{\log m}{s}\right)$ for all $\vect{j} \in [m]^d$.
Therefore, by \holder inequality, we have \begin{align}
&\abs{\psi_{x,h}(x)^\top \left(\Psi_{x,h}^\top \Psi_{x,h}\right)^{-1}\begin{pmatrix}
\ldots \\
\triangle_\vect{j} + \expect{\eta_{\vect{j}}}\\
\ldots
\end{pmatrix}} \nonumber \\
 \le &\norm{\psi_{x,h}(x)^\top \left(\Psi_{x,h}^\top \Psi_{x,h}\right)^{-1}}_1 \norm{\begin{pmatrix}
\ldots \\
\triangle_\vect{j} + \expect{\eta_{\vect{j}}}\\
\ldots
\end{pmatrix}}_{\infty} \nonumber\\
\le &C_1\left(\frac{L}{m}+\epsilon+\frac{\log m}{s}\right). \label{eqn:lpr_biased}
\end{align} for some constant $C_1 > 0$.

Next,  notice that \begin{align*}
\psi_{x,h}(x)^\top \left(\Psi_{x,h}^\top \Psi_{x,h}\right)^{-1}Z_{t,h}\begin{pmatrix}
\ldots \\
f\left(\frac{\vect{j}}{m}\right) + \triangle_\vect{j} + \left(\eta_\vect{j}-\expect{\eta_{\vect{j}}\right)}\\
\ldots
\end{pmatrix} - f(x) 
\end{align*} is the standard error term for non-parametric estimation with unbiased stochastic noise using local polynomial regression~\citep{nemirovski2000topics}.
Since we have $\variance\left(\eta_j\right) \le \frac{C}{s}$, we obtain the following bound \begin{align}
\int_{[0,1]^d} \abs{\psi_{x,h}(x)^\top \left(\Psi_{x,h}^\top \Psi_{x,h}\right)^{-1}Z_{t,h}\begin{pmatrix}
\ldots \\
f\left(\frac{\vect{j}}{m}\right) + \triangle_\vect{j} + \left(\eta_\vect{j}-\expect{\eta_{\vect{j}}\right)}\\
\ldots
\end{pmatrix} - f(x)}^2 \le C_L n^{-\frac{2\beta}{2\beta+d}} \label{eqn:lpr_unbiased}
\end{align} for some $C_L > 0$ depending on $L$ if we choose  $h\asymp \left(nL^2\right)^{-\frac{1}{2\beta+d}}$.
Lastly, plugging in $m\asymp \sqrt{n}/{\sqrt[4]{\log n}} $ and combining Equation~\eqref{eqn:lpr_biased} and~\eqref{eqn:lpr_unbiased} we obtain the desired result.

\subsection{Proof of Theorem \ref{thm:lower_bound_for_holder}}

The statistical rate $\wp(L)\cdot n^{-2\beta/(2\beta+d)}$ can be established by standard minimax lower bound arguments for (non-robust) nonparametric regression problems (see e.g. \citep{tsybakov2009introduction}).
We shall therefore focus solely on establishing the contamination dependency $\epsilon^2$ in this proof.

Given a function $f$, the observation model can be re-formulated as (the robust version) of a $n$-dimensional Gaussian random vector with mean $\left[f(x_{\vect{i}})\right]$ for $\vect{i} \in [p]^d$ and identity covariance.
	In light of Theorem 5.1 of~\citep{chen2015robust}, we only need to show there exists two functions $f_1,f_2\in\Lambda(\beta,L)$ or $\Sigma\left(\beta,p,L\right)$ such that the total variation of the following two distributions
	\begin{align*}
	D_1 = N(\left[f(x_{\vect{i}})\right],I), \quad D_2 = N(\left[f(x_{\vect{i}})\right],I)
	\end{align*} for $\vect{i} \in [p]^d$
	is upper bounded by $\epsilon/(1-\epsilon)$.
	Then by Pinsker's inequality we have $\|{f_1-f_2}\|_2^2 \asymp  4\text{TV}^2\left(D_1,D_2\right) \ge 4\epsilon^2$.
	The existence of such function pairs $f_1,f_2\in\Lambda(\beta,L)$ or $\Sigma\left(\beta,p,L\right)$ is easily satisfied by choosing two constant function making the total variation equal to $\epsilon/(1-\epsilon)$.
	The theorem is hence proved.

\end{document}